\documentclass[
10pt
]{article}

\usepackage[%
left=1in,
right=1in,
bottom=1in,
top=1in,
footskip=0.3in,
]{geometry}

\usepackage{color}
\definecolor{LinkColor}{rgb}{0,0,1}
\definecolor{LinkColor2}{rgb}{1,0,0}
\definecolor{lbcolor}{rgb}{0.85,0.85,0.85}
\definecolor{FrameColor}{rgb}{0.85,0.85,0.85}

\usepackage{lineno}

\newcommand*\patchAmsMathEnvironmentForLineno[1]{%
	\expandafter\let\csname old#1\expandafter\endcsname\csname #1\endcsname
	\expandafter\let\csname oldend#1\expandafter\endcsname\csname end#1\endcsname
	\renewenvironment{#1}%
	{\linenomath\csname old#1\endcsname}%
	{\csname oldend#1\endcsname\endlinenomath}}% 
\newcommand*\patchBothAmsMathEnvironmentsForLineno[1]{%
	\patchAmsMathEnvironmentForLineno{#1}%
	\patchAmsMathEnvironmentForLineno{#1*}}%
\AtBeginDocument{%
	\patchBothAmsMathEnvironmentsForLineno{equation}%
	\patchBothAmsMathEnvironmentsForLineno{align}%
	\patchBothAmsMathEnvironmentsForLineno{flalign}%
	\patchBothAmsMathEnvironmentsForLineno{alignat}%
	\patchBothAmsMathEnvironmentsForLineno{gather}%
	\patchBothAmsMathEnvironmentsForLineno{multline}%
}

\usepackage[ngerman, english]{babel}
\usepackage[utf8]{inputenc}
\usepackage{enumerate}
\usepackage[normalem]{ulem}
\usepackage{setspace}
\usepackage{multicol}
\usepackage[babel,french=guillemets,german=swiss]{csquotes}
\usepackage{caption}
\usepackage{calligra}
\usepackage{array}
\usepackage{booktabs}
\usepackage{framed}
\usepackage{rotating}
\usepackage{longtable}
\usepackage{multirow}
\usepackage{tabularx}
\newcolumntype{L}[1]{>{\raggedright\arraybackslash}p{#1}} % linksb?ndig mit Breitenangabe
\newcolumntype{C}[1]{>{\centering\arraybackslash}p{#1}} % zentriert mit Breitenangabe
\newcolumntype{R}[1]{>{\raggedleft\arraybackslash}p{#1}} % rechtsb?ndig mit Breitenangabe

\usepackage{graphicx}
\usepackage{amsmath}
\usepackage{amssymb}
\usepackage{dsfont}
\usepackage{nicefrac}
\usepackage{empheq}
\allowdisplaybreaks
\usepackage{amsthm}
\usepackage{upgreek}

\usepackage[%
pdftitle={Titel},% 
pdfauthor={Autor},% 
pdfcreator={LaTeX, LaTeX with hyperref and KOMA-Script},% Genutzte Programme
pdfsubject={Betreff}, % Betreff
pdfkeywords={Keywords}
]{hyperref} 

\hypersetup{%
	colorlinks	=true,% Definition der Links im PDF File
	linkcolor	=LinkColor,%
	anchorcolor	=LinkColor,%
	citecolor	=LinkColor2,%
	filecolor	=LinkColor,%
	menucolor	=LinkColor,%
	pagecolor	=LinkColor,%
	urlcolor	=LinkColor,%
}

\newtheorem{theorem}{Theorem}
\newtheorem{lemma}[theorem]{Lemma}
\newtheorem{proposition}[theorem]{Proposition}
\newtheorem{corollary}[theorem]{Corollary}
\newtheorem{definition}[theorem]{Definition}
\newtheorem{remark}[theorem]{Remark}

\newcommand{\norm}[1]{\ensuremath\lVert #1 \rVert}
\newcommand{\norml}[2]{\ensuremath\lVert #2 \rVert_{L^{#1}}}
\newcommand{\normL}[2]{\ensuremath\lVert #2 \rVert_{\L^{#1}}}
\newcommand{\normw}[3]{\ensuremath\lVert #2 \rVert_{W^{#1,#2}}}
\newcommand{\normW}[3]{\ensuremath\lVert #2 \rVert_{\W^{#1,#2}}}
\newcommand{\normh}[2]{\ensuremath\lVert #2 \rVert_{H^{#1}}}
\newcommand{\normH}[2]{\ensuremath\lVert #2 \rVert_{\H^{#1}}}

\def\R{\mathbb R}
\def\CR{\mathcal R}
\def\N{\mathbb N}
\def\n{\mathbf n}
\def\a{\mathbf a}
\def\b{\mathbf b}

\def\U{\mathbb U}

\def\f{\mathbf{f}}
\def\UR{\mathbb{U}_R}

\def\u{{\bar u}}
\def\L{\mathbf L}
\def\H{\mathbf H}
\def\h{\mathds{h}}
\def\v{\mathbf{v}}
\def\w{\mathbf{w}}
\def\T{\mathbf{T}}
\def\I{\mathbf{I}}
\def\D{{\mathbf{D}}}

\def\V{{\mathcal V}}
\def\W{{\mathbf W}}
\def\F{{\mathbf{F}}}
\def\q{{\mathbf{q}}}

\def\tu{{\tilde u}}
\def\tv{{\tilde\v}}
\def\tvarphi{{\tilde\varphi}}
\def\tmu{{\tilde\mu}}
\def\tsigma{{\tilde\sigma}}
\def\tp{{\tilde p}}

\def\tw{{\tilde\w}}
\def\tphi{{\tilde\vartheta}}
\def\ttau{{\tilde\tau}}
\def\trho{{\tilde\rho}}

\def\OmegaT{{\Omega_T}}
\def\GammaT{{\Gamma_T}}

\def\intO{\int_\Omega}
\def\intOT{\int_\OmegaT}
\def\intG{\int_\Gamma}
\def\intGT{\int_\GammaT}

\def\ddt{\frac{\mathrm d}{\mathrm dt}}

\def\dx{\;\mathrm dx}

\def\dt{\;\mathrm dt}

\def\dxt{\;\mathrm d(x,t)}

\def\dS{\;\mathrm dS}

\def\del{\partial}

\def\delt{\partial_{t}}

\def\deln{\partial_\n}

\def\d{{\mathrm{d}}}

\def\grad{\nabla}
\def\laplace{\Delta}
\def\divergence{\textnormal{div}}

\def\sminus{\hspace{-2pt}-\hspace{-2pt}}
\def\scdot{\hspace{-2pt}\cdot\hspace{-2pt}}
\def\scolon{\hspace{-2pt}:\hspace{-2pt}}

\def\tand{\quad\text{and}\quad}
\def\twith{\quad\text{with}\quad}

\def\scdot{\hspace{-2pt}\cdot\hspace{-2pt}}

\def\wto{\rightharpoonup}
\def\wsto{\overset{*}{\rightharpoonup}}

\def\itema{\item[\textnormal{(a)}]}
\def\itemb{\item[\textnormal{(b)}]}
\def\itemc{\item[\textnormal{(c)}]}

\def\itemi{\item[\textnormal{(i)}]}
\def\itemii{\item[\textnormal{(ii)}]}
\def\itemiii{\item[\textnormal{(iii)}]}

\begin{document}
	
%
%	TITLE PAGE
%
	
\begin{center}	
	\LARGE{Optimal medication for tumors modeled by a Cahn-Hilliard-Brinkman equation}
\end{center}

\bigskip

\begin{center}	
	\normalsize{Matthias Ebenbeck and Patrik Knopf \footnote{orcid.org/0000-0003-4115-4885}}\\[1mm]
	\textit{Department of Mathematics, University of Regensburg, 93053 Regensburg, Germany}\\[1mm]
	\texttt{Matthias.Ebenbeck@ur.de, Patrik.Knopf@ur.de}\\[5mm]
	\textit{This is a preprint version of the paper. Please cite as:}\\
	M. Ebenbeck and P. Knopf, Calc. Var. (2019) 58: 131.\\
	\url{https://doi.org/10.1007/s00526-019-1579-z}
\end{center}

\bigskip

\begin{abstract}
	In this paper, we study a distributed optimal control problem for a diffuse interface model for tumor growth. The model consists of a Cahn-Hilliard type equation for the phase field variable coupled to a reaction diffusion equation for the nutrient and a Brinkman type equation for the velocity. The system is equipped with homogeneous Neumann boundary conditions for the tumor variable and the chemical potential, Robin boundary conditions for the nutrient and a \grqq no-friction" boundary condition for the velocity. The control acts as a medication by cytotoxic drugs and enters the phase field equation. The cost functional is of standard tracking type and is designed to track the variables of the state equation during the evolution and the distribution of tumor cells at some fixed final time. We prove that the model satisfies the basics for calculus of variations and we establish first-order necessary optimality conditions for the optimal control problem.
	\\[1ex]
	\\[1ex]
\end{abstract}

\bigskip

%
%	INTRODUCTION
%

\section{Introduction}

The development of tumors involves many different biological and chemical factors. Since cancer arises due to disturbances in both cell growth and development, knowing the underlying processes will not only help to cure the disease but also provide an understanding of the mechanisms concerned with life itself. The usage of new techniques in cell and molecular biology applied to human tumors provides valuable insights. However, in order to study processes which cannot be easily observed by experiments, further methods have to be developed. In the recent past, mathematical models for tumor growth have turned out to be promising since some of them compare well with clinical experiments, see \cite{AgostiEtAl,BearerEtAl,FrieboesEtAl}. \\[1ex]
In particular, multiphase models, describing the tumor as a saturated medium, have gained much more interest. These models are typically based on mass and momentum balance equations and mass/momentum exchange between the different phases and can be closed by appropriate constitutive laws. Several biological mechanisms like chemotaxis, mitosis, angiogenesis or necrosis can be incorporated or effects due to stress, plasticity or viscoelasticity can be included, see \cite{AstaninPreziosi,GarckeLamNuernbergSitka,OdenTinsleyHawkins,PreziosiTosin}. If the mixture is assumed to consist of only two phases, many contributions in the literature cover Cahn-Hilliard type equations coupled to reaction-diffusion equations for an unknown species acting as a nutrient (e.g. oxygen or glucose), see \cite{ColliGilardiHilhorst,FrigeriGrasselliRocca,GarckeLam2,GarckeLam3,HawkinsZeeKristofferOdenTinsley,HilhorstKampmannNguyenZee}. Some of these models also include velocity via a Darcy law or a Stokes/Brinkman equation, see \cite{EbenbeckGarcke,GarckeLam1,GarckeLam4,JiangWuZheng}.\\[1ex]
In this paper, we consider the following model: Let $\Omega\subset \R^d$ with $d=2,3,$ be a bounded domain. For a fixed final time $T>0$, we write $\OmegaT:=\Omega\times(0,T)$. By $\n $ we denote the outer unit normal on $\del\Omega$ and $\deln g \coloneqq \grad g\cdot \n $ denotes the outward normal derivative of the function $g$ on $\Gamma$. Our state system is given by
\begin{subequations}
	\label{EQ:CHB}
	\begin{empheq}[left=\text{(CHB)}\empheqlbrace]{align}
		\label{EQ:CHB1}
		\divergence(\v) &= (P\sigma-A)\h(\varphi) &&\text{in}\; \OmegaT,\\
		\label{EQ:CHB2}
		-\divergence(\T(\v,p)) + \nu\v &= (\mu+\chi\sigma)\grad\varphi &&\text{in}\; \OmegaT,\\
		\label{EQ:CHB3}
		\delt\varphi + \divergence(\varphi\v) &= m\laplace\mu + (P\sigma-A-u)\h(\varphi) &&\text{in}\; \OmegaT,\\
		\label{EQ:CHB4}
		\mu &= -\epsilon\laplace\varphi + \epsilon^{-1}\psi'(\varphi) -\chi\sigma &&\text{in} \; \OmegaT,\\
		\label{EQ:CHB5}
		-\laplace\sigma + \h(\varphi)\sigma &= 0 &&\text{in}\; \OmegaT,\\[2mm]
		\label{EQ:CHB6}
		\deln\mu=\deln\varphi &= 0 &&\text{in}\; \GammaT,\\
		\label{EQ:CHB7}
		\deln\sigma &= K(1-\sigma) &&\text{in}\; \GammaT,\\
		\label{EQ:CHB8}
		\T(\v,p)\n &= 0 &&\text{in}\; \GammaT,\\[2mm]
		\label{EQ:CHB9}
		\varphi(0) &= \varphi_0	&&\text{in}\; \Omega,
	\end{empheq}
\end{subequations}
where the viscous stress tensor is defined by
\begin{equation}
\label{definition_stress_tensor}\T(\v,p) \coloneqq 2\eta \D\v+\lambda\divergence(\v)\I  - p\I ,
\end{equation}
and the symmetric velocity gradient is given by
\begin{equation*}
\D\v\coloneqq \tfrac{1}{2}(\grad\v+\grad\v^T).
\end{equation*}
In \eqref{EQ:CHB}, $\v$ denotes the volume-averaged velocity of the mixture, $p$ denotes the pressure and $\sigma$ denotes the concentration of an unknown species acting as a nutrient. The function $\varphi$ denotes the difference of the local relative volume fractions of tumor tissue and healthy tissue where $\{x\in\Omega:\varphi(x) = 1\}$ represents the region of pure tumor tissue, $\{x\in\Omega:\varphi(x) = -1\}$ stands for the surrounding pure healthy tissue and $\{x\in\Omega: -1< \varphi(x) <1\}$ represents the transition between these pure phases. Furthermore, $\mu$ denotes the chemical potential associated with $\varphi$. The positive constant $m$ represents the mobility for the phase variable $\varphi$. The thickness of the diffuse interface is modelled by a small parameter $\epsilon>0$ and the constant $\nu>0$ is related to the fluid permeability. Moreover, the constants $\eta$ and $\lambda$ are non-negative and represent the shear and the bulk viscosity, respectively. The proliferation rate $P$, the apoptosis rate $A$ and the chemotaxis parameter $\chi$ are non-negative constants, whereas $K$ is a positive permeability constant.
The term $-u\h(\varphi)$ in \eqref{EQ:CHB3} models the elimination of tumor cells by cytotoxic drugs and the function $u$ will act as our control. This specific control term has been investigated in \cite{GarckeLamRocca} where a simpler model was studied in which the influence of the velocity $\v$ is neglected. Since it does not play any role in the analysis, we set $\epsilon = 1$.
\\[1ex]
We investigate the following distributed optimal control problem:
\begin{align}
\notag\text{Minimize}\quad I(\varphi,\mu,\sigma,\v,p,u)&:= 
\frac{\upalpha_0} 2 \norml{2}{\varphi(T)-\varphi_f}^2
+ \frac{\upalpha_1} 2 \norm{\varphi-\varphi_d}_{L^2(\Omega_T)}^2
+ \frac{\upalpha_2} 2\norm{\mu-\mu_d}_{L^2(\Omega_T)}^2 \\
&\qquad + \frac{\upalpha_3} 2\norm{\sigma-\sigma_d}_{L^2(\Omega_T)}^2
+ \frac{\upalpha_4} 2\norm{\v-\v_d}_{L^2(\Omega_T)}^2
+ \frac{\kappa} 2\norm{u}_{L^2(\Omega_T)}^2
\end{align}
subject to the control constraint
\begin{align}
u\in \U:=\big\{ u\in L^2(L^2) \;\big\vert\; a(x,t) \le u(x,t) \le b(x,t) \;\;\text{for almost every}\; (x,t)\in\OmegaT \big\}
\end{align}
for box-restrictions $a,b\in L^2(L^2)$ with $a\le b$ almost everywhere in $\OmegaT$ and the state system (CHB). Here, $\upalpha_0,...,\upalpha_4$ and $\kappa$ are nonnegative constants. The optimal control problem can be interpreted as the search for a strategy how to supply a medication based on cytotoxic drugs such that
\begin{itemize}
	\itemi a desired evolution of the tumor cells, chemical potential, nutrient concentration and velocity (expressed by the target functions $\varphi_d$, $\mu_d$, $\sigma_d$ and $\v_d$) is realized as good as possible;
	\itemii a therapeutic target (expressed by the final distribution $\varphi_f$) is achieved in the best possible way;
	\itemiii the amount of supplied drugs does not cause harm to the patient (expressed by both the control constraint and the last term in the cost functional).
\end{itemize}
The ratio between the parameters $\upalpha_0,...,\upalpha_4$ and $\kappa$ can be adjusted according to the importance of the individual therapeutic targets. 
\\[1ex]
The uncontrolled version of \eqref{EQ:CHB} is a simplification of the model considered in \cite{EbenbeckGarcke2} when making the following choices in the notation used therein:
\begin{alignat*}{3}
\eta(\cdot)&\equiv \eta &\qquad \lambda(\cdot)&\equiv\lambda,\\
\Gamma_\v(\varphi,\sigma)& = (P\sigma-A)\h(\varphi)&\qquad\Gamma_{\varphi}(\varphi,\sigma)&= (P\sigma-A)\h(\varphi)\\
m(\cdot) &\equiv m&\sigma_{\infty}&=1.
\end{alignat*}
Typically, the function $\h(\cdot)$ is nonnegative and interpolates between $\h(-1)=0$ and $\h(1)=1$. Source terms of this type are, for instance, studied in \cite{GarckeLamSitkaStyles}. Typical choices for $\h$ are discussed in Remark \ref{REM:ASS}(c). \\[1ex] The choice $\sigma_\infty=1$ can be justified by a non-dimensionalization argument (see \cite{ChristiniLiLowengrubWise}).  
\\[1ex]
As \eqref{EQ:CHB7} is equivalent to $\sigma = 1 - K^{-1}\deln \sigma$, we obtain a Dirichlet condition in the limit $K\to\infty$. In \cite[Thm. 2.5]{EbenbeckGarcke2} it was rigorously shown that weak solutions of the uncontrolled version of \eqref{EQ:CHB} converge to a weak solution of the corresponding system with \eqref{EQ:CHB8} replaced by $\sigma = 1$ as $K$ tends to infinity. 
\\[1ex]
Before analyzing the optimal control problem we have to establish the fundamental requirements for calculus of variations. In Section 3, we prove the existence of a control-to-state operator that maps any admissible control $u\in\U$ onto a corresponding unique strong solution of the state equation (CHB). Furthermore, we show that this control-to-state operator is Lipschitz-continuous, Fréchet differentiable and satisfies a weak compactness property. 
\\[1ex]
Eventually, in Section 4, we investigate the above optimal control problem. First, we show that there exists at least one globally optimal solution. After that, we present necessary conditions for local optimality which can be formulated concisely by means of adjoint variables. These conditions are of great importance for possible numerical implementations as they provide the foundation for many computational optimization methods. 
\\[1ex]
Finally, we also want to refer to some further works where optimal control problems for tumor models are studied.
Results on optimal control problems for tumor models based on ODEs are investigated in \cite{SchaettlerLedzewicz2,Oke,SchaettlerLedzewicz,Swan}. In the context of PDE-based control problems we refer to \cite{Benosman} where a tumor growth model of advection-reaction-diffusion type is considered.
There are various papers analyzing optimal control problems for Cahn-Hilliard equations (e.g., \cite{ColliFarshbaf,HintermuellerWegner, Sprekels-Wu}). Furthermore, control problems for the convective Cahn-Hilliard equation where the control acts as a velocity were investigated in \cite{ColliGilardiSprekels,GilardiSprekels,ZhaoLiu1,ZhaoLiu2} whereas in \cite{Biswas} the control enters in the momentum equation of a Cahn-Hilliard-Navier-Stokes system. As far as control problems for Cahn-Hilliard-based models for tumor growth are considered, there are only a few contributions where an equation for the nutrient is included in the system. In \cite{ColliGilardiRoccaSprekels}, the authors investigated an optimal control problem consisting of a Cahn-Hilliard-type equation coupled to a time-dependent reaction-diffusion equation for the nutrient, where the control acts as a right-hand side in this nutrient equation. The model they considered was firstly proposed in \cite{HawkinsZeeKristofferOdenTinsley} and later well-posedness and existence of strong solutions were established in \cite{FrigeriGrasselliRocca}. However, effects due to velocity are not included in their model and mass conservation holds for the sum of tumor and nutrient concentrations. 
Moreover, we want to mention the papers \cite{GarckeLamRocca} and \cite{Cavaterra} where optimal control problems of treatment time are studied. In \cite{GarckeLamRocca} the control enters the phase field equation in the same way as ours whereas in \cite{Cavaterra} it enters the nutrient equation. Although the nutrient equation in both papers is non-stationary, some of the major difficulties do not occur since the velocity is assumed to be negligible ($\v=\mathbf 0$). 
A regularized version of the model in \cite{Cavaterra} was studied in \cite{Signori} for a logarithmic-type potential and in \cite{Signori2} for a broad class of regular potentials.

\section{Preliminaries}
We first want to fix some notation: For any (real) Banach space $X$, its corresponding norm is denoted by $\norm{\cdot}_X$. $X^\ast$ denotes the dual space of $X$ and $\langle \cdot{,}\cdot \rangle_X$ stands for the duality pairing between $X^\ast$ and $X$. If $X$ is an inner product space, its inner product is denoted by $(\cdot{,}\cdot)_X$. We define the scalar product of two matrices by
\begin{equation*}
\mathbf{A}:\mathbf{B}\coloneqq \sum_{j,k=1}^{d}a_{jk}b_{jk}\quad\text{for } \mathbf{A}=(a_{ij})_{1\le i,j\le d},\; \mathbf{B}=(b_{ij})_{1\le i,j\le d} \in\R^{d\times d}.
\end{equation*}
For the standard Lebesgue and Sobolev spaces with $1\leq p\leq \infty$, $k>0$,  we use the notation $L^p\coloneqq L^p(\Omega)$ and $W^{k,p}\coloneqq W^{k,p}(\Omega)$ with norms $\norml{p}{\cdot}$ and $\normw{k}{p}{\cdot}$ respectively.  In the case $p=2$ we write $H^k\coloneqq W^{k,2}$ and the norm $\normh{k}{\cdot}$. By $\mathbf{L}^p$, $\mathbf{W}^{k,p}$ and $\mathbf{H}^k$, we denote the corresponding spaces of vector or matrix valued functions. 
For any Banach space $X$, $p\in [1,\infty]$, $k\in\N$ and $T>0$ we use the notation
\begin{align*}
	L^p(X) \coloneqq L^p(0,T;X),\qquad W^{k,p}(X)\coloneqq W^{k,p}(0,T;X), \qquad C(X)\coloneqq C([0,T];X).
\end{align*}
For the dual space $X^\ast$ of a Banach space $X$, we introduce the (generalized) mean value by 
\begin{equation*}
v_{\Omega}\coloneqq \frac{1}{|\Omega|}\intO v\dx\quad\text{for } v\in L^1, \quad v_{\Omega}^\ast\coloneqq \frac{1}{|\Omega|}\langle v{,}1\rangle_X\quad\text{for } v\in X^\ast.
\end{equation*}
Moreover, we introduce the function spaces
\begin{equation*}
L_0^2\coloneqq \{w\in L^2\colon w_{\Omega}=0\},\quad (H^1)_0^\ast\coloneqq \{f\in (H^1)^\ast\colon f_{\Omega}^\ast =0\},\quad H_\n^2\coloneqq \{w\in H^2\colon \deln w = 0 \text{ on } \del\Omega\}.
\end{equation*}
Then, the Neumann-Laplace operator $-\laplace_N\colon H^1\cap L_0^2\to (H^1)_0^\ast$ is positive definite and self-adjoint. In particular, by the Lax-Milgram theorem and the Poincaré inequality,
%(see (\ref{Poincare}))
the inverse operator $(-\laplace_N)^{-1}\colon (H^1)_0^\ast\to H^1\cap L_0^2$ is well-defined, and we set $u\coloneqq (-\laplace_N)^{-1}f$ for ${f\in (H^1)_0^\ast}$ if $u_{\Omega}=0$ and 
\begin{equation*}
-\laplace u = f\text{ in }\Omega,\quad\deln u =0\text{ on }\del\Omega.
\end{equation*}
We have dense and continuous embeddings $H_\n^2\hookrightarrow H^1\hookrightarrow L^2\simeq (L^2)^\ast\hookrightarrow (H^1)^\ast\hookrightarrow (H_\n^2)^\ast$ and the identifications $\langle u{,}v\rangle_{H^1}=(u{,}v)_{L^2}$, $\langle u{,}w\rangle _{H^2} = (u{,}w)_{L^2}$ for all $u\in L^2,~ v\in H^1$ and $w\in H_\n^2$.\newline
Furthermore, we define the function spaces
\begin{align*}
\V_1 &\coloneqq \big(H^1(L^2)\cap L^{\infty}(H^2)\cap L^2(H^4)\big) \times \big(L^{\infty}(L^2)\cap L^2(H^2)\big)\times \big(H^1(H^1)\cap C(H^1)\cap L^{\infty}(H^2)\big)\\
&\quad \times  L^8(\H^2)\times L^8(H^1),\\
\V_2 &\coloneqq \big(H^1((H^1)^*)\cap L^{\infty}(H^1)\cap L^2(H^3)\big)\times L^2(H^1)\times L^2(H^1)\times L^2(\H^1)\times L^2(L^2),\\
\V_3&\coloneqq L^2(L^2)\times H^1\times L^2(H^1)\times L^2(L^2)\times L^{2}(\H^1),
\end{align*}
endowed with the standard norms.
\\[1ex]
Finally, we state the following lemma (see \cite{AbelsTerasawa} for a proof), which will be needed later for the construction of solutions by a Galerkin ansatz: 
\begin{lemma}\label{Stokes_Neumann result}
	Let $\Omega\subset\R^d, d=2,3,$ be a bounded domain with $C^{1,1}$-boundary and outer unit normal $\mathbf{n}$ and $1<q<\infty$. Furthermore, assume that $g\in W^{1,q}$ and $\mathbf{f}\in \mathbf{L}^q$ and $\eta,\nu,\lambda$ are constants fulfilling $\eta,\nu>0,~\lambda\geq 0$. Then, there exists a unique solution $(\v,p)\in \mathbf{W}^{2,q}\times W^{1,q}$ of the system 
	\begin{subequations}
		\label{Stokes_subsystem}
		\begin{alignat}{3}\
		\label{Stokes_subsystem_1}-\divergence(2\eta \D\v +\lambda\divergence(\v)\I)+\nu\v + \grad p &= \mathbf{f}&&\quad \text{a.e. in }\Omega,\\
		\label{Stokes_subsystem_2}\divergence(\v) &= g&&\quad\text{a.e. in }\Omega,\\
		\label{Stokes_subsystem_3}(2\eta \D\v +\lambda\divergence(\v)\I-p\I)\mathbf{n} &= \mathbf{0}&&\quad \text{a.e. on }\del \Omega,
		\end{alignat}
	\end{subequations}
	satisfying the following estimate
	\begin{equation}
	\label{Stokes_Neumann_estimate}\normW{2}{q}{\v} + \normw{1}{q}{p}\leq C\left(\normL{q}{\f}+\normw{1}{q}{g}\right),
	\end{equation}
	with a constant $C$ depending only on $\eta,\lambda,\nu,q$ and $\Omega$.
\end{lemma}

\pagebreak[2]

\noindent Throughout this paper, we make the following assumptions:
\\[1ex]
\noindent\textbf{Assumptions.}
\begin{enumerate}
	\item[(A1)] The domain $\Omega\subset\R^d,~d=2,3,$ is bounded with $C^4-$boundary $\Gamma:=\del\Omega$ and the initial datum $\varphi_0\in H^2_\n$ is prescribed. 
	\item[(A2)] The constants $T$, $K$, $\eta$, $\nu$, $m$ are positive and the constants $P$, $A$, $\lambda$, $\chi$ are nonnegative.
	\item[(A3)] The nonnegative function $\h$ belongs to $C^2_b(\R)$, i.e., $\h$ is bounded, twice continuously differentiable and its first and second-order derivatives are bounded. Without loss of generality, we assume that $|\h|\le 1$.
	\item[(A4)] The function $\psi\in C^3(\R)$ is non-negative and can be written as
	\begin{equation}
	\label{assumption_on_psi_1}\psi(s) = \psi_1(s) + \psi_2(s)\quad\forall s\in \R,
	\end{equation}
	where $\psi_1, \psi_2\in C^3(\R)$ and 
	\begin{alignat}{3}
	\label{assumption_on_psi_2}R_1(1+|s|^{\rho-2})&\leq \psi_1''(s)\leq R_2(1+|s|^{\rho-2})&&\quad\forall s\in\R,\\
	\label{assumption_on_psi_3}|\psi_2''(s)|&\leq R_3&&\quad\forall s\in\R,\\
	\label{assumption_on_psi_4}|\psi'''(s)|&\leq R_4(1+|s|^3)&&\quad\forall s\in \R,
	\end{alignat} 
	where $R_i,~i=1,...,4$, are positive constants with $R_1<R_2$ and $\rho\in [2,6]$. Furthermore, if $\rho=2$, we assume that $2R_1>R_3$.
	\item[(A5)] The target functions satisfy $(\varphi_d,\varphi_f,\mu_d,\sigma_d,\v_d) \in \V_3$.
\end{enumerate}

\begin{remark} $\;$
	\label{REM:ASS}
	\begin{itemize}
		\itema 	Using \textnormal{(A4)}, it is straightforward to check that there exist positive constants $R_i,~i=5,...,9,$ such that 
		\begin{align}
		\label{assumption_on_psi_5}\psi(s)&\geq R_5|s|^{\rho} - R_6 &&\quad \forall s\in \R, \\
		\label{assumption_on_psi_6}|\psi'(s)|&\leq R_7(1+|s|^{\rho-1}) &&\quad \forall s\in \R, \\
		\label{assumption_on_psi_7}|\psi'(s_1)-\psi'(s_2)|&\leq R_8(1+|s_1|^4 + |s_2|^4)|s_1-s_2|&&\quad \forall s_1,s_2\in\R,\\
		\label{assumption_on_psi_8}|\psi''(s_1)-\psi''(s_2)|&\leq R_9(1+|s_1|^3 + |s_2|^3)|s_1-s_2|&&\quad \forall s_1,s_2\in\R.
		\end{align}
		\itemb  The assumptions \eqref{assumption_on_psi_2}-\eqref{assumption_on_psi_4} (and thus also \eqref{assumption_on_psi_5}-\eqref{assumption_on_psi_8}) are fulfilled by the classical double-well potential $\psi(s)=\frac 1 4 (s^2-1)^2$. For the splitting we can choose $\psi_1(s)=\frac 1 4 s^4$ and $\psi_2(s)=-\frac 1 2 s^2 + \frac 1 4 $.
		\itemc For the function $\h(\cdot)$, there are two choices which are quite popular in the literature. In, e.g., \cite{GarckeLam3, GarckeLamSitkaStyles}, the choice for $\h$ is given by
		\begin{equation*}
		\h(\varphi) = \max\Big\{0,\min\left\{1,\tfrac{1}{2}(1+\varphi)\right\}\Big\}\quad\forall \varphi\in\R,
		\end{equation*}
		satisfying $\h(-1)= 0$, $\h(1)=1$. Other authors preferred to assume that $\h$ is only active on the interface, i.e., for values of $\varphi$ between $-1$ and $1$, which motivates functions of the form
		\begin{equation*}
		\h(\varphi) = \max\Big\{0,\tfrac{1}{2}(1-\varphi^2)\Big\}\quad\text{or}\quad \h(\varphi) = \tfrac{1}{2}\big(\cos\big(\pi \min \big\{1,\max\{\varphi,-1\} \big\}\big) + 1\big),
		\end{equation*}
		see, e.g., \cite{HilhorstKampmannNguyenZee,KahleLam}. Surely, we would have to use regularized versions of these choices to fulfill \textnormal{(A3)}.
	\end{itemize}
\end{remark}

\section{The control-to-state operator}

In this section, we consider the equation (CHB) as presented in the introduction. First, we define a certain set of admissible controls that are suitable for the later approach. We will see that each of these admissible controls induces a unique strong solution (the so-called state) of the system (CHB). Therefore, we can define a control-to-state-operator which maps any admissible control onto its corresponding state. We show that this operator has some important properties that are essential for calculus of variations: It is Lipschitz-continuous, Fréchet-differentiable and satisfies a weak compactness property.

\subsection{The set of admissible controls}
The set of admissible controls is defined as follows:
\begin{definition}
	Let $a,b\in L^2(L^2)$ be arbitrary fixed functions with $a\le b$ almost everywhere in $\OmegaT$. Then the set
	\begin{align}
		\U:=\big\{ u\in L^2(L^2) \;\big\vert\; a(x,t) \le u(x,t) \le b(x,t) \;\;\text{for almost every}\; (x,t)\in\OmegaT \big\}
	\end{align}
	is referred to as the \textbf{set of admissible controls}. Its elements are called \textbf{admissible controls}.
\end{definition}
Note that this box-restricted set of admissible controls $\U$ is a non-empty, bounded subset of the Hilbert space $L^2(L^2)$ since for all $u\in\U$,
\begin{align}
	\norm{u}_{L^2(L^2)} < \norm{a}_{L^2(L^2)} + \norm{b}_{L^2(L^2)} + 1 =: R.
\end{align}
This means that 
\begin{align}
	\U \subsetneq \UR \twith \UR := \big\{ u\in L^2(L^2) \;\big\vert\; \norm{u}_{L^2(L^2)} <R \big\}.
\end{align}
Obviously, the set $\U$ is also convex and closed in $L^2(L^2)$. Therefore, it is weakly sequentially compact (see \cite[Thm.\,2.11]{troeltzsch}). 

\subsection{Strong solutions and uniform bounds}

We can show that the system (CHB) has a unique strong solution for every control $u\in\UR$:

\begin{theorem}
	\label{THM:EXS}
	Let $ u\in\UR$ and $\varphi_0\in H^2_\n(\Omega)$ be arbitrary. Then, the system \textnormal{(CHB)} has a unique strong solution $(\varphi_u,\mu_u,\sigma_u,\v_u,p_u)$ with
	\begin{gather*}
	\varphi_u \in H^1(L^2) \cap L^\infty(H^2)\cap L^2(H^4), \quad 	\mu_u \in L^\infty(L^2)\cap L^2(H^2), \\
	\sigma_u \in H^1(H^1)\cap C(H^1)\cap L^{\infty}(H^2),\quad \v_u \in L^8(\H^2)  , \quad p_u \in L^8(H^1). 
	\end{gather*}
	This unique solution is called the \textbf{state} of the control system. Moreover, the strong solution quintuple $(\varphi_u,\mu_u,\sigma_u,\v_u,p_u)$ satisfies the following bounds that are uniform in $u$: 
	\begin{gather*}
		\norm{\varphi_u}_{H^1(L^2)\cap L^\infty(H^2)\cap L^2(H^4)} \le C_1, \quad \norm{\mu_u}_{L^\infty(L^2)\cap L^2(H^2)} \le C_2, \\
		\norm{\sigma_u}_{H^1(H^1)\cap C(H^1)\cap L^{\infty}(H^2)} \le C_3,\quad \norm{\v_u}_{L^8(\H^2) } \le C_4 \quad \norm{p_u}_{L^2(H^1)} \le C_5,	
	\end{gather*}
	where $C_1,...,C_5>0$ are constants that depend only on the system paramters and on $R$, $\Omega$, $\Gamma$ and $T$.
\end{theorem}

\begin{proof}
	The assertion follows with slight modifications in the proof of \cite[Thm.\,2.12]{EbenbeckGarcke2}. Indeed, testing (\ref{EQ:CHB3}) with $\mu+\chi\sigma$, it turns out that we have to estimate an additional term given by $-\intO u\h(\varphi)(\mu+\chi\sigma)\dx$. Using Hölder's, Young's and Poncaré's inequalities, we obtain
	\begin{align}
	\label{proof_ex_1}
	&\nonumber\left|\intO u\h(\varphi)(\mu+\chi\sigma)\dx\right| = \left|\intO u\h(\varphi)\big((\mu + \chi\sigma-(\mu_{\Omega}-\chi\sigma_{\Omega})\big)\dx + (\mu_{\Omega}+\chi\sigma_{\Omega})\intO u\h(\varphi)\dx\right|\\
	&\quad \leq \frac{1}{4\delta}\norml{\infty}{\h(\varphi)}^2\norml{2}{u}^2 + \delta C_P^2\normL{2}{\grad(\mu+\chi\sigma)}^2
		+ |\mu_{\Omega}+\chi\sigma_{\Omega}|\,\norml{\infty}{\h(\varphi)}\norml{2}{1}\norml{2}{u},
	\end{align}
	for $\delta>0$ arbitrary, where $C_P$ is the constant arising in Poincaré's inequality. Testing (\ref{EQ:CHB4}) with $1$, using (\ref{EQ:CHB6}) and the assumptions on $\psi(\cdot)$, we obtain
	\begin{equation}
	\label{proof_ex_2}|\mu_{\Omega}+\chi\sigma_{\Omega}|\leq C(1+\norml{1}{\psi(\varphi)}).
	\end{equation}
	Plugging in (\ref{proof_ex_2}) into (\ref{proof_ex_1}), using the boundedness of $\h(\cdot)$ and Young's inequality, we obtain
	\begin{equation}
	\label{proof_ex_3}\left|\intO u\h(\varphi)(\mu+\chi\sigma)\dx\right| \leq C_{\delta}(1+\norml{2}{u}^2)(1+\norml{1}{\psi(\varphi)}) + 2\delta C_P^2\left(\normL{2}{\grad\mu}^2 + \chi^2\normL{2}{\grad\sigma}^2\right).
	\end{equation}
	Then, the first term on the right-hand side of this equation can be controlled via a Gronwall argument, whereas the last two terms can be absorbed into the left-hand side of an energy identity. To obtain strong solutions, (\ref{EQ:CHB3}) is tested by $\delt\varphi$, leading to an additional term $\intO u\h(\varphi)\laplace^2\varphi\dx$. Applying Hölder's and Young's inequalities yields
	\begin{equation}
	\label{proof_ex_4}\left|\intO u\h(\varphi)\delt\varphi\dx\right|\leq \norml{2}{u}\norml{\infty}{\h(\varphi)}\norml{2}{\delt\varphi}\leq \frac{1}{4\delta}\norml{2}{u}^2\norml{\infty}{\h(\varphi)}^2 + \delta \norml{2}{\delt\varphi}^2,
	\end{equation}
	for $\delta>0$ arbitrary. Hence, the last term on the right-hand side of this inequality can again be absorbed when choosing $\delta>0$ sufficiently small, whereas the first term can be controlled due to the assumptions on $\h(\cdot)$ and since $\norm{u}_{L^2(L^2)}\leq R$ for all $u\in\UR$. Apart from the estimates needed to deduce \eqref{proof_ex_3} and \eqref{proof_ex_4}, the remaining arguments are exactly the same as in the proof of \cite[Thm.\,2.12]{EbenbeckGarcke2}.
\end{proof}	

\begin{corollary}
	\label{COR:INF}
	Let $u\in\UR$ and $\varphi_0\in H^2_\n(\Omega)$ be arbitrary and let $(\varphi_u,\mu_u,\sigma_u,\v_u,p_u)$ denote the strong solution of the system \textnormal{(CHB)}. Then $\varphi_u$ has the following additional properties:
	\begin{gather*}
		 \varphi_u\in C(H^2), \quad \varphi_u\in C(\overline{\OmegaT}) \twith \norm{\varphi_u}_{C(H^2)\cap C(\overline{\OmegaT})} \le C_6 %, \quad \varphi_u(T) \in H^2_\n(\Omega)
	\end{gather*}
	for some constant $C_6>0$ that depends only on the system parameters and on $R$, $\Omega$, $\Gamma$ and $T$.
\end{corollary}

\begin{proof}
	First, recall that $H^4$ is continuously embedded in $L^2$. Using a result from interpolation theory \cite[c.\,III,\,Thm.\,4.10.2]{amann} we can conclude that $H^1(L^2) \cap L^2(H^4) \hookrightarrow C\big([0,T];(L^2,H^4)_{\frac 1 2,2}\big)$ where $(L^2,H^4)_{\frac 1 2,2}$ denotes the real interpolation space between $L^2$ and $H^4$ of type $\frac 1 2$. According to \cite[s.\,4.3.1,\,Thm.\,1]{triebel}, it holds that $(L^2,H^4)_{\frac 1 2,2} = H^2$ which proves the first assertion. As $H^2$ is continuously embedded in $C(\overline{\Omega})$, it directly follows that $\varphi_u\in C(\overline{\OmegaT})$ with
	\begin{align*}
		\norm{\varphi_u}_{C(\overline\OmegaT)} = \norm{\varphi_u}_{L^\infty(L^\infty)} \le C_0 \, \norm{\varphi_u}_{L^\infty(H^2)} \le C_0C_1=:C_6
	\end{align*}
	for some constant $C_0\ge 0$ that depends only on $R$, $\Omega$, $\Gamma$ and $T$. This means that the second assertion is established. 
\end{proof}

Due to Theorem \ref{THM:EXS}, we can define an operator that maps any admissible control onto its corresponding state:
\begin{definition}
	\label{DEF:CSO}
	For any $u\in\UR$ we write $(\varphi_u,\mu_u,\sigma_u,\v_u,p_u)$ to denote its induced strong solution of \textnormal{(CHB)} given by Theorem \ref{THM:EXS}.
	Then the operator
	\begin{align*}
		S:\;\UR \to \V_2,\quad u\mapsto S(u):=(\varphi_u,\mu_u,\sigma_u,\v_u,p_u)
	\end{align*}
	is called the \textbf{control-to-state} operator. To be precise, it holds that $S(\UR)\subset \V_1\subset \V_2$.
\end{definition}

\noindent\textbf{Comment.} The control-to-state operator is defined not only for admissible controls but for all controls in $\UR$. This will be especially important in subsection 3.4 because Fréchet differentiability is merely defined for open subsets of $L^2(L^2)$. Unlike the open ball $\UR$, the set $\U$ is closed and its interior is empty. Therefore it makes sense to investigate the control-to-state operator on the open superset $\UR$ instead. 
\\[1ex]
In the following subsections, some properties of the control-to-state operator will be established that are essential for the treatment of optimal control problems.

\subsection{Lipschitz continuity}

The first important property of the control-to-state operator is Lipschitz continuity with respect to the control $u\in\U$. This is given by the following lemma:
\begin{lemma}
	\label{LEM:LIP}
	The control-to-state operator $S$ is Lipschitz continuous in the following sense: 
	\begin{enumerate}
		\itema There exists a constant $L_1>0$ depending only on the system parameters and on $R$, $\Omega$, $\Gamma$ and $T$ such that for all $u,\tu\in\U$:
		\begin{align}
		\label{IEQ:L1}
		\norm{S(u)-S(\tu)}_{\V_2}\leq L_1\norm{u-\tu}_{L^2(L^2)}. 
		\end{align}
%		\begin{align*}
%		\norm{\varphi_u-\varphi_\tu}_{H^1((H^1)^*)\cap L^\infty(H^1)\cap L^2(H^3)}+
%		\norm{\mu_u-\mu_\tu}_{L^2(H^1)} +\norm{\sigma_u-\sigma_\tu}_{L^2(H^1)}\quad & \\
%		+\norm{\v_u-\v_\tu}_{L^2(\H^1)}  +\norm{p_u-p_\tu}_{L^2(L^2)}&\leq L_1\norm{u-\tu}_{L^2(L^2)}. 
%		\end{align*}
		\itemb There exists a constant $L_2> 0$ depending only on $R$, $\Omega$, $\Gamma$ and $T$ such that for all $u,\tu\in\U$:
		\begin{align}
		\label{IEQ:L2}
			\norm{\varphi_u-\varphi_\tu}_{L^\infty(\OmegaT)} &\le L_2\norm{u-\tu}_{L^2(L^2)}^{1/2}.
		\end{align}
	\end{enumerate}
\end{lemma}

\begin{proof}
	Let $C$ denote a generic nonnegative constant that depends only on $R$, $\Omega$, $\Gamma$ and $T$ and may change its value from line to line. For brevity, we set
	\begin{align*}
	\u := u-\tu
	\tand
	(\bar\varphi,\bar\mu,\bar\v,\bar\sigma,\bar p):=(\varphi_{u},\mu_{u},\v_{u},\sigma_{u},p_{u})-(\varphi_\tu,\mu_\tu,\v_\tu,\sigma_\tu,p_\tu).
	\end{align*}
	In \cite[Thm.\,2.8]{EbenbeckGarcke2} it has already been established that a solution of the Cahn-Hilliard-Brinkman equation without the control term $u\h(\varphi)$ depends Lipschitz continuously on the initial and boundary data. The proof of item (a) proceeds similarly since the following inequalities hold: For any $\delta>0$ and all $u,\tu\in\U$,
	\begin{align}
		\label{EQ:LIP1}
		\big|\big( u\h(\varphi_u) - \tu\h(\varphi_\tu){,}\bar\varphi \big)_{L^2}\big| 
			&\le C\,\norml{2}{\bar u}^2 + C\, \norml{2}{\bar \varphi}^2 + C\,\norml{2}{\tu}\normh{1}{\bar\varphi}^2,\\
		\label{EQ:LIP2}
		\big|\big( u\h(\varphi_u) - \tu\h(\varphi_\tu){,}\laplace\bar\varphi \big)_{L^2}\big| 
			&\le C\delta^{-1}\, \norml{2}{\bar u}^2 + C\delta^{-1}\, \norml{2}{\tu}^2 \normh{1}{\bar\varphi}^2
				+ 2\delta \norml{2}{\laplace\bar\varphi}^2 + \delta \normL{2}{\grad\laplace\bar\varphi}^2.
	\end{align}
	To prove \eqref{EQ:LIP1} and \eqref{EQ:LIP2} we use that $u\h(\varphi_u) - \tu\h(\varphi_\tu) = \u\h(\varphi_u) + \tu\big(\h(\varphi_u) - \h(\varphi_\tu)\big)$. From $\norm{\h(\varphi_u)}_{L^\infty} \le C$ we deduce that
	\begin{align}
		\label{EQ:LIP3}
		|( \u\h(\varphi_u) {,} \bar\varphi )_{L^2}| &\le C\,\left( \norm{\bar u}_{L^2}^2\, +  \norm{\bar \varphi}_{L^2}^2\right)
	\end{align}
	and, by Young's inequality with $\delta>0$,
	\begin{align}
		\label{EQ:LIP4}
		|( \u\h(\varphi_u) {,} \laplace\bar\varphi )_{L^2}| &\le C\delta^{-1}\, \norml{2}{\bar u}^2 + \delta \norml{2}{\laplace\bar\varphi}^2.
	\end{align}
	Moreover, we have
	\begin{align}
		\label{EQ:LIP5}
		\big|\big( \tu\big(\h(\varphi_u) - \h(\varphi_\tu)\big){,} \bar\varphi \big)_{L^2}\big| 
			\le C\, \norml{2}{\tu} \, \norml{4}{\bar\varphi}^2
			\le C\, \norml{2}{\tu} \, \normh{1}{\bar\varphi}^2
	\end{align}
	and, using interpolation and Young's inequality with $\delta>0$, we obtain that
	\begin{align}
		&\big|\big( \tu\big(\h(\varphi_u) - \h(\varphi_\tu)\big){,} \laplace\bar\varphi \big)_{L^2}\big| 
			\le C\, \norml{2}{\tu} \, \norml{6}{\bar\varphi} \, \norml{3}{\laplace\bar\varphi} \notag\\
		&\quad \le C\, \norml{2}{\tu} \, \normh{1}{\bar\varphi} \, \normh{1}{\laplace\bar\varphi} 
			\le C\delta^{-1}\, \norml{2}{\tu}^2 \normh{1}{\bar\varphi}^2 
			+ \delta \norml{2}{\laplace\bar\varphi}^2 + \delta \normL{2}{\grad\laplace\bar\varphi} ^2.
		\label{EQ:LIP6}
	\end{align}
	Now, \eqref{EQ:LIP1} follows from \eqref{EQ:LIP3} and \eqref{EQ:LIP5} while \eqref{EQ:LIP2} follows from \eqref{EQ:LIP4} and \eqref{EQ:LIP6}. For the proof of (b) recall that $\norm{\varphi_u-\varphi_\tu}_{L^\infty(H^2)} \le C$. Hence, by the Gagliardo-Nirenberg inequality,
	\begin{align*}
		&\norm{\varphi_u-\varphi_\tu}_{L^\infty(L^\infty)} 
			\le C\,\norm{\varphi_u-\varphi_\tu}_{L^\infty(L^6)}^{1/2} \, \norm{\varphi_u-\varphi_\tu}_{L^\infty(H^2)}^{1/2} \\
		&\quad\le C\,\norm{\varphi_u-\varphi_\tu}_{L^\infty(H^1)}^{1/2} \, \norm{\varphi_u-\varphi_\tu}_{L^\infty(H^2)}^{1/2} 
			\le L_2\,\norm{u-\tu}_{L^2(L^2)}^{1/2} 
	\end{align*}
	if the constant $L_2$ is chosen suitably.
\end{proof}	

\subsection{A weak compactness property}

As the control-to-state operator is nonlinear, the following result will be essential to prove existence of an optimal control (see Section 4.1):
\begin{lemma}
	\label{LEM:COMP}
	Suppose that $(u_k)_{k\in\N}\subset\U$ is converging weakly in $L^2(L^2)$ to some weak limit $\u\in\U$. Then
		\begin{alignat*}{6}
		&\varphi_{u_k} &&\wto \varphi_\u\quad &&\text{in}\; H^1(L^2)\cap L^2(H^4),\quad
		&&\varphi_{u_k} &&\to \varphi_\u &&\text{in}\; C(W^{1,r})\cap C(\overline{\OmegaT}),~r\in [1,6),\\
		&\mu_{u_k}&&\wto\mu_\u\quad &&\text{in}\; L^2(H^2),
		&&\v_{u_k}&&\wto\v_\u\quad &&\text{in}\; L^2(H^2),\\
		&\sigma_{u_k}&&\wto\sigma_\u\quad &&\text{in}\; L^2(H^2), \quad  
		&&p_{u_k}&&\wto p_\u\quad &&\text{in}\; L^2(H^1),
		\end{alignat*}
after extraction of a subsequence, where the limit $(\varphi_\u,\mu_\u,\sigma_\u,\v_\u,p_\u)$ is the strong solution of \textnormal{(CHB)} to the control $\u \in\U$.
\end{lemma}

\noindent\textbf{Comment.} This result actually means weak compactness of the control-to-state operator restricted to $\U$ since any bounded sequence in $\U$ has a weakly convergent subsequence according to the Banach-Alaoglu theorem. However, this property can not be considered as weak continuity as the extraction of a subsequence is necessary.

\begin{proof}
	Using the uniform bounds that were established in Theorem \ref{THM:EXS} and standard compactness arguments, we can conclude that there exist functions $\varphi$, $\v$, $\mu$, $\sigma$ and $p$ having the desired regularity such that 
	\begin{gather*}
	\varphi_{u_k} \wsto \varphi\quad \text{in}\; H^1(L^2)\cap L^{\infty}(H^2)\cap L^2(H^4),\quad \varphi_{u_k} \wto \varphi\quad \text{in}\; H^1(L^2)\cap L^2(H^4),\quad
	\mu_{u_k}\wto\mu\quad \text{in}\; L^2(H^2),\\
	\sigma_{u_k}\wto\sigma\quad \text{in}\; L^2(H^2),\quad
	\v_{u_k}\wto\v\quad\text{in}\; L^2(H^2), \quad  
	p_{u_k}\wto p\quad \text{in}\; L^2(H^1)
	\end{gather*}
	up to a subsequence. The Aubin-Lions lemma (cf. \cite[Sec.\,8, Cor. 4]{Simon}) implies that $H^1(L^2)\cap L^\infty(H^2)$ is compactly embedded in the space $C(W^{1,r}),~r\in[1,6),$ and thus the convergence $\varphi_{u_k} \to \varphi$ in $C(W^{1,r}),~r\in[1,6),$ directly follows after subsequence extraction. In particular, by the Sobolev embedding $W^{1,r}\subset C(\bar{\Omega}),~r\in (3,6)$, we obtain that $\varphi_{u_k}\to \varphi$ in $C(\overline{\OmegaT})$, whence
	\begin{equation}
	\label{lemma_9_conv}\psi'(\varphi_{u_k})\to\psi'(\varphi)\quad\text{and}\quad\h(\varphi_{u_k}) \to \h(\varphi) \quad \text{in}\; C(\overline{\Omega_T}), \quad k\to\infty.
	\end{equation}
	It remains to show that the quintuple $(\varphi,\v,\mu,\sigma,p)$ is a strong solution of the system (CHB) according to the control $u$ which means it is equal to $(\varphi_\u,\v_\u,\mu_\u,\sigma_\u,p_\u)$. Due to the above convergence result, all linear terms in (CHB) are converging weakly in $L^2(L^2)$ to their respective limit. The nonlinear terms must be treated individually. 
	From \eqref{lemma_9_conv}, we can easily conclude that
	\begin{align*}
	\sigma_{u_k}\h(\varphi_{u_k}) \wto \sigma\h(\varphi) \tand u_k\h(\varphi_{u_k}) \wto \u\h(\varphi) \quad \text{in}\; L^2(\OmegaT), \quad k\to\infty,
	\end{align*}
	since $\norm{\h(\varphi)}_{L^\infty(\OmegaT)}\le C$, $\norm{u_k}_{L^2(L^2)}<R$ and $\norm{\sigma_{u_k}}_{L^2(L^2)}\le C_4$. Recalling that $\varphi_{u_k}\to\varphi$ in $C(W^{1,r})$ $\cap C(\overline{\OmegaT})$ and $\v_{u_k}\wto \v$ in $L^2(H^2)$ as $k\to\infty$, by the product of weak-strong convergence it follows that
	\begin{align*}
	\divergence(\varphi_{u_k}\v_{u_k}) \wto \divergence(\varphi\v) \quad \text{in}\; L^2(\OmegaT), \quad k\to\infty.
	\end{align*}
	Now, let $\zeta\in C_0^\infty(\OmegaT)$ be arbitrary. Then, since $C_0^{\infty}(\OmegaT)\subset L^2(\OmegaT)$, we obtain
	\begin{align*}
	&\intOT \big(\delt\varphi - \divergence(\varphi\v) -m\laplace\mu - (P\sigma-A-\u)\h(\varphi)\big)\,\zeta \dxt \\
	&\quad = \underset{k\to\infty}{\lim} \intOT \big(\delt\varphi_{u_k} - \divergence(\varphi_{u_k}\v_{u_k}) -m\laplace\mu_{u_k} - (P\sigma_{u_k}-A-u_k)\h(\varphi_{u_k})\big)\,\zeta \dxt = 0
	\end{align*}
	and consequently, 
	\begin{align*}
	\delt\varphi - \divergence(\varphi\v) = m\laplace\mu + (P\sigma-A-\u)\h(\varphi)\quad \text{almost everywhere in}\;\OmegaT.
	\end{align*}
	We proceed analogously with the remaining equations of (CHB). This proves that $(\varphi,\mu,\sigma,\v,p)$ is a strong solution of the system (CHB) to the control $\u$ and thus, because of uniqueness, we have $(\varphi,\mu,\sigma,\v,p)=(\varphi_\u,\mu_\u,\sigma_\u,\v_\u,p_\u)$ almost everywhere in $\OmegaT$. 
\end{proof}	
\subsection{The linearized system}  

We want to show that the control-to-state operator is also Fréchet differentiable on the open ball $\UR$ (and therefore especially on its strict subset $\U$). Since the Fréchet derivative is a linear approximation of the control-to-state operator at some certain point $u\in\UR$, it will be given by a linearized version of (CHB):
\begin{subequations}
	\label{EQ:LIN}
\begin{empheq}[left=\textnormal{(LIN)}\empheqlbrace]{align}
\label{EQ:LIN1}
\divergence(\v) &= P\sigma\h(\varphi_u) + (P\sigma_u-A)\h'(\varphi_u)\varphi + F_1&&\text{in}\; \OmegaT,\\
\label{EQ:LIN2}
-\divergence(\T(\v,p)) + \nu\v &= (\mu_u+\chi\sigma_u)\grad\varphi + (\mu+\chi\sigma)\grad\varphi_u + \F &&\text{in}\; \OmegaT,\\
\label{EQ:LIN3}
\delt\varphi + \divergence(\varphi_u\v) + \divergence(\varphi\v_u) &= m\laplace\mu + (P\sigma_u - A - u)\h'(\varphi_u)\varphi \notag\\
&\qquad + P\sigma\h(\varphi_u) + F_2&&\text{in}\; \OmegaT,\\
\label{EQ:LIN4}
\mu &= -\laplace\varphi + \psi''(\varphi_u)\varphi - \chi\sigma + F_3 &&\text{in}\; \OmegaT,\\
\label{EQ:LIN5}
-\laplace\sigma + \h'(\varphi_u)\varphi\sigma_u + \h(\varphi_u)\sigma &=  F_4 &&\text{in}\; \OmegaT,\\[2mm]
\label{EQ:LIN6}
\deln\mu=\deln\varphi &= 0 &&\text{in}\; \GammaT,\\
\label{EQ:LIN7}
\deln\sigma &= -K\sigma &&\text{in}\; \GammaT,\\
\label{EQ:LIN8}
\T(\v,p)\n &= 0 &&\text{in}\; \GammaT,\\[2mm]
\label{EQ:LIN9}
\varphi(0) &= 0	&&\text{in}\; \Omega,
\end{empheq}
\end{subequations}
where $F_i:\OmegaT\to\R,\,1\le i\le 4$ and $\F:\OmegaT\to\R^3$ are given functions that will be specified later on. A weak solution of this linearized system is defined as follows:
\begin{definition}
	Let $u\in\UR$ be arbitrary. Then a quintuple $(\varphi,\mu,\sigma,\v,p)$ is called a weak solution of \textnormal{(LIN)} if
	it lies in $\V_2$ and satisfies the equations
		\begin{alignat}{2}
			\label{WF:LIN0}
			\varphi(0) &= 0 &&\quad \text{a.e. in}\; \Omega, \\
			\label{WF:LIN1}
			\divergence(\v) &= P\sigma\h(\varphi_u) + (P\sigma_u-A)\h'(\varphi_u)\varphi + F_1 &&\quad \text{a.e. in}\; \OmegaT
		\end{alignat}
		and
		\begin{align}
			\label{WF:LIN2}
			&\intO \T(\v,p):\grad\tv + \nu\v\cdot\tv \dx = \intO \big[(\mu_u+\chi\sigma_u)\grad\varphi + (\mu+\chi\sigma)\grad\varphi_u + \F\big]\cdot\tv\dx, \\
			\nonumber
			&\big\langle \delt\varphi , \tvarphi \big\rangle_{H^1} = \intO \hspace{-5pt}  -m\grad\mu\scdot\grad\tvarphi 
				+ \big[(P\sigma_u - A - u)\h'(\varphi_u)\varphi+ P\sigma\h(\varphi_u) + F_2\big]\tvarphi \,\dx \\
			\label{WF:LIN3} &\qquad\qquad\qquad -\intO \big(\divergence(\varphi_u\v) +\divergence(\varphi\v_{u})\big)\tilde{\varphi}\d x,\\
			\label{WF:LIN4}
			&\intO \mu \tmu \dx = \intO \grad\varphi\cdot\grad\tmu + \big[\psi''(\varphi_u)\varphi - \chi\sigma + F_3\big]\tmu \dx, \\
			\label{WF:LIN5}
			&\intO \grad\sigma\cdot\grad\tsigma \dx + K \intG \sigma\tsigma \dS = -\intO \big[\h'(\varphi_u)\varphi\sigma_u + \h(\varphi_u)\sigma\big]\tsigma\dx,
		\end{align}
		for almost every $t\in(0,T)$ and all $\tvarphi,\tmu,\tsigma \in H^1,~\tv \in \H^1$.
\end{definition}

\begin{lemma}
	\label{LEM:EST}
	Let $u\in\UR$ be any control and let $(\varphi_u,\mu_u,\sigma_u,\v_u,p_u)$ denote its corresponding state. Moreover, let $\F\in L^2(\L^2)$, $F_1,F_3\in L^2(H^1)$ and $F_2,F_4\in L^2(L^2)$ be arbitrary and let us assume that the quintuple $(\varphi,\mu,\sigma,\v,p)$ is any weak solution of the system \textnormal{(LIN)}. Then, there exists some constant $C>0$ depending only on the system parameters and on $R$, $\Omega$, $\Gamma$ and $T$ such that:
	\begin{align}
	\label{EST:LIN}
	\norm{(\varphi,\mu,\sigma,\v,p)}_{\V_2} \le C\left(\norm{\grad F_3}_{L^2(L^2)} 
	+ \norm{\F}_{L^2(\L^2)}  + \sum_{i=1}^4 \norm{F_i}_{L^2(L^2)}  \right).
	\end{align}
\end{lemma}

We do not give a detailed proof of this lemma but merely refer to the proof of Proposition \ref{PROP:LIN} where we show that the approximate solutions constructed by a Galerkin scheme satisfy the inequality \eqref{EST:LIN} (see \eqref{lin_eq_24}). The estimates established in this approach can be carried out in the same way for any weak solution $(\varphi,\mu,\sigma,\v,p)$ of (LIN) which proves the above lemma. 
\\[1ex]
Existence and uniqueness of weak solutions to the system (LIN) is established by the following proposition:
\begin{proposition}
	\label{PROP:LIN}
	Let $u\in\UR$ be any control and let $(\varphi_u,\mu_u,\sigma_u,\v_u,p_u)$ denote its corresponding state. Moreover, let $\F\in L^2(\L^2)$, $F_1,F_3\in L^2(H^1)$ and $F_2,F_4\in L^2(L^2)$ be arbitrary. Then the system \textnormal{(LIN)}
	has a unique weak solution $(\varphi,\mu,\sigma,\v,p)$. 
%	Moreover, there exists some constant $C>0$ depending only on the system parameters and on $R$, $\Omega$, $\Gamma$ and $T$ such that:
%	\begin{align}
%	%\label{EST:LIN}
%		\norm{(\varphi,\mu,\sigma,\v,p)}_{\V_2} \le C\left(\norm{\grad F_3}_{L^2(L^2)} 
%		+ \norm{\F}_{L^2(\L^2)}  + \sum_{i=1}^4 \norm{F_i}_{L^2(L^2)}  \right).
%	\end{align}
\end{proposition}

\begin{proof}
\textit{Step 1: Galerkin approximation.}
We will construct approximate solutions by applying a Galerkin approximation with respect to $\varphi$ and $\mu$ and at the same time solve for $\sigma,\,\v$ and $p$ in the corresponding whole function spaces. As Galerkin basis for $\varphi$ and $\mu$, we will use the eigenfunctions of the Neumann-Laplace operator $\{w_i\}_{i\in\N}$ that form a Schauder basis of $L^2$. We will choose $w_1 = 1$.
By elliptic regularity, we see that $w_i\in H_\n^2\cap H^4$ and for every $g\in H_\n^2$ with $g_k\coloneqq\sum_{i=1}^{k}(g{,}w_i)_{L^2}w_i$ we obtain
\begin{equation*}
\laplace g_k = \sum_{i=1}^{k}(g{,}w_i)_{L^2}\laplace w_i = -\sum_{i=1}^{k}(g{,}\lambda_i w_i)_{L^2}w_i = \sum_{i=1}^{k}(g{,}\laplace w_i)_{L^2}w_i = \sum_{i=1}^{k}(\laplace g{,}w_i)_{L^2}w_i,
\end{equation*}
where $\lambda_i$ is the corresponding eigenvalue to $w_i$. Therefore, $\laplace g_k$ converges strongly to $\laplace g$ in $L^2$. Again using elliptic regularity theory, we obtain that $g_k$ converges strongly to $g$ in $H_\n^2$. Thus the eigenfunctions $\{w_i\}_{i\in \N}$ of the Neumann-Laplace operator form an orthonormal Schauder basis in $L^2$ which is also a Schauder basis of $H_\n^2$. 
We fix $k\in\N$ and define $\mathcal{W}_k\coloneqq \text{span}\{w_1,...,w_k\}$. Our aim is to find functions of the form 
\begin{equation*}
\varphi_k(t,x)=\sum_{i=1}^{k}a_i^k(t)w_i(x),~~\mu_k(t,x)=\sum_{i=1}^{k}b_i^k(t)w_i(x)
\end{equation*}
satisfying for all $v\in\mathcal{W}_k$ the following approximation problem
\begin{subequations}
	\begin{align}
	\nonumber\intO\delt\varphi_k v\dx &= \intO-m\grad\mu_k\cdot\grad v +\big((P\sigma_u-A-u)\h'(\varphi_u)\varphi_k + P\sigma_k\h(\varphi_u) + F_2\big)v\dx\\
	\label{approx_problem_eq_1a}&\quad -\intO \big(\divergence(\varphi_u\v_k) + \divergence(\varphi_k\v_u)\big)v\dx,\\
	\label{approx_problem_eq_1b}\intO \mu_kv\dx &= \intO \grad\varphi_k\cdot \grad v+ \big(\psi''(\varphi_u)\varphi_k - \chi\sigma_k + F_3\big)v\dx,\\
	\label{approx_problem_eq_initial}\varphi_{k}(0,\cdot) &= 0,
	\end{align}
	where the nutrient concentration $\sigma_k$ is defined as the unique solution of
	\begin{alignat}{3}
	\label{approx_problem_eq_1c}0&=-\laplace\sigma_k + \h'(\varphi_u)\varphi_k\sigma_u + \h(\varphi_u)\sigma_k -  F_4 &&\quad\text{in }\Omega,\\
	\label{approx_problem_eq_1cc}\del_\n \sigma_k &= -K\sigma_k&&\quad \text{on }\Gamma,
	\end{alignat}
	 and the velocity $\v_k$ and the pressure $p_k$ are defined as the unique solution of (\ref{Stokes_subsystem}) with 
	\begin{alignat}{3}
	\label{approx_problem_eq_1d}-\divergence(\T(\v_k,p_k))+\nu\v_k &= (\mu_u + \chi\sigma_u)\grad\varphi_k + (\mu_k+\chi\sigma_k)\grad\varphi_u + \F &&\quad\text{a.e. in }\Omega,\\
	\label{approx_problem_eq_1e}\divergence(\v_k) &= P\sigma_k\h(\varphi_u) + (P\sigma_u-A)\h'(\varphi_u)\varphi_k + F_1&&\quad\text{a.e. in }\Omega,\\
	\label{approx_problem_eq_1f}\T(\v_k,p_k)\mathbf{n} &= \mathbf{0}&&\quad\text{a.e. on }\del\Omega.
	\end{alignat}
\end{subequations}
Using the continuous embedding $H_\n^2\hookrightarrow L^{\infty}$, the assumptions on $\h(\cdot), \F, F_1$ and Theorem \ref{THM:EXS}, it is straightforward to verify that 
\begin{align*}
&\big((\mu_u + \chi\sigma_u)\grad\varphi_k + (\mu_k+\chi\sigma_k)\grad\varphi_u + \F \big) \in L^2 \tand
&\big(P\sigma_k\h(\varphi_u) + (P\sigma_u-A)\h'(\varphi_u)\varphi_k + F_1\big)\in H^1,
\end{align*}
almost everywhere in $(0,T)$. Therefore, by Lemma \ref{Stokes_Neumann result}, we obtain that $(\v_k,p_k)\in \H^2\times H^1$ for almost every $t\in (0,T)$ and (\ref{approx_problem_eq_1d})-(\ref{approx_problem_eq_1f}) is fulfilled almost everywhere in the respective sets. After some straightforward calculations, it can be verified that (\ref{approx_problem_eq_1a})-(\ref{approx_problem_eq_1f}) is equivalent to a linear system of $k$ ODEs in the $k$ unknown $(a_1^k,...,a_k^k)^T\eqqcolon \a_k$. Due to the assumptions on $\psi(\cdot),~\h(\cdot)$, the stability of \eqref{approx_problem_eq_1c}-\eqref{approx_problem_eq_1cc} and (\ref{approx_problem_eq_1d})-(\ref{approx_problem_eq_1f}) under perturbations and Theorem \ref{THM:EXS}, the theory of ODEs (see e.g. \cite[Chap. 3, Problem 1]{CoddingtonLevinson}) yields the existence of a unique $\a_k\in W^{1,1}((0,T);\R^k)$ for each $k\in\N$. 
\\[1ex]
Using elliptic regularity theory, we obtain $\sigma_k\in L^2(H^2)$ as the unique strong solution of \eqref{approx_problem_eq_1c}-\eqref{approx_problem_eq_1cc} and then we define $\b_k\coloneqq (b_1^k,...,b_k^k)^T$ using (\ref{approx_problem_eq_1b}). Hence, the Galerkin scheme yields the existence of a unique solution triple $(\varphi_k,\mu_k,\sigma_k)\in (W^{1,1}(0,T;H_\n^2\cap H^4))^2 \times L^2(H^2)$. Finally, we can define $(\v_k,p_k)$ as the unique solution of the subsystem (\ref{approx_problem_eq_1d})-(\ref{approx_problem_eq_1f}) and, with similar arguments as above, it follows that $\v_k(t) \in\H^2$ and $p_k(t)\in H^1$ for almost every $t\in [0,T]$. We remark that $(\varphi_k,\mu_k)\in C(H_\n^2\cap H^4)^2$ and (\ref{approx_problem_eq_1a})-(\ref{approx_problem_eq_1b}), (\ref{approx_problem_eq_1c})-(\ref{approx_problem_eq_1cc}), (\ref{approx_problem_eq_1d})-(\ref{approx_problem_eq_1f}) are fulfilled almost everywhere in $(0,T)$. 
\\[1ex]
In the following, we will show a-priori-estimates for the solutions of the approximating system (\ref{approx_problem_eq_1a})-(\ref{approx_problem_eq_1f}). Therefore, we use a generic constant $C$ which may change it's value from one line to another, but has to be independent of $k\in \N$.
\\[1ex]
\textit{Step 2:}
Applying \cite[Sec.\,III.3, Ex. 3.16]{Galdi}, there exists a solution $\w_k\in \H^1$ of
\begin{alignat*}{3}
\divergence(\w_k) &= P\sigma_k \h(\varphi_u) + (P\sigma_u - A) \h'(\varphi_u)\varphi_k + F_1&&\quad\text{in }\Omega,\\
\mathbf{w_k} &= \left(\frac{1}{|\partial \Omega|}\int_{\Omega}P\sigma_k \h(\varphi_u) + (P\sigma_u - A) \h'(\varphi_u)\varphi_k + F_1\dx\right)\mathbf{n}&&\quad\text{on }\del \Omega,
\end{alignat*}
satisfying
\begin{equation}
\label{lin_eq_0} \normH{1}{\w_k}\leq C\norml{2}{P\sigma_k \h(\varphi_u) + (P\sigma_u - A) \h'(\varphi_u)\varphi_k + F_1}.
\end{equation}
Then, multiplying (\ref{approx_problem_eq_1d}) with $\v_k-\w_k$, choosing $v= a_i^k(\lambda_iw_i + w_i)$ in (\ref{approx_problem_eq_1a}), $v= ma_i^k\lambda_i(\lambda_iw_i + w_i)$ in (\ref{approx_problem_eq_1b}), summing the resulting identities over $i=1,...,k$, integrating by parts and adding the resulting equations, we obtain
\begin{align}
\nonumber&\ddt\frac{1}{2}\intO |\varphi_k|^2 + |\grad\varphi_k|^2\dx + m\intO |\laplace\varphi_k|^2+|\grad\laplace\varphi_k|^2\dx + \intO 2\eta|\D\v_k|^2+\nu|\v_k|^2\dx \\
\nonumber &\quad= m\intO \grad (\psi''(\varphi_u)\varphi_k -\chi\sigma_k+ F_3)\cdot\grad\laplace\varphi_k\dx+ \intO (\divergence(\varphi_u\v_k)+\divergence(\varphi_k\v_u)-F_2)\laplace\varphi_k\dx\\
\nonumber &\qquad - \intO \big((P\sigma_u-A-u)\h'(\varphi_u)\varphi_k + P\sigma_k\h(\varphi_u)\big)\laplace\varphi_k \dx + m\intO (\psi''(\varphi_u)\varphi_k-\chi\sigma_k + F_3)\laplace\varphi_k\dx\\
\nonumber &\qquad + \intO \big((P\sigma_u-A-u)\h'(\varphi_u)\varphi_k + P\sigma_k\h(\varphi_u) + F_2-\divergence(\varphi_u\v_k)-\divergence(\varphi_k\v_u)\big)\varphi_k\dx\\
\label{lin_eq_1} &\qquad + \intO \big((\mu_u+\chi\sigma_u)\grad\varphi_k + (\mu_k+\chi\sigma_k)\grad\varphi_u+\F\big)\cdot(\v_k-\w_k)+  2\eta\D\v_k\colon\grad\w_k + \nu\v_k\cdot\w_k\dx,
\end{align} 
where we used (\ref{approx_problem_eq_1e})-(\ref{approx_problem_eq_1f}).
In what follows, we will estimate the terms on the right-hand side of (\ref{lin_eq_1}) individually and frequently use Hölder's and Young's inequalities. \newline Due to the boundedness of $\psi''(\varphi_u),\psi'''(\varphi_u)\in L^{\infty}(\OmegaT)$ and $\grad\varphi_u\in L^{\infty}(\L^6)$ and the Sobolev embedding $H^1\subset L^3(\Omega)$, we calculate
\begin{align}
\nonumber &\left|m\intO \grad(\psi''(\varphi_u)\varphi_k)\cdot\grad\laplace\varphi_k\dx\right| = \left|\intO(\psi'''(\varphi_u)\varphi_k\grad\varphi_u+\psi''(\varphi_u)\grad\varphi_k)\cdot\grad\laplace\varphi_k\dx\right|\\
&\quad \leq C\left(\norml{3}{\varphi_k}\normL{6}{\grad\varphi_u} + \normL{2}{\grad\varphi_k}\right)\normL{2}{\grad\laplace\varphi_k}
\label{lin_eq_2} \leq C\normh{1}{\varphi_k}^2+\tfrac{m}{16}\normL{2}{\grad\laplace\varphi_k}^2.
\end{align}
For the next two terms, we obtain
\begin{equation}
\label{lin_eq_3} \left|\intO m \grad(F_3-\chi\sigma_k)\cdot\grad\laplace\varphi_k\dx\right|\leq C\left(\normL{2}{\grad F_3}^2 + \normL{2}{\grad\sigma_k}^2\right) + \tfrac{m}{16}\normL{2}{\grad\laplace\varphi_k}^2.
\end{equation}
Since $\deln\varphi_k = 0$ on $\GammaT$, we know that
\begin{equation}
\label{lin_eq_3a}\norml{2}{\laplace\varphi_k}\leq \normL{2}{\grad\varphi_k}^{1/2}\normL{2}{\grad\laplace\varphi_k}^{1/2}.
\end{equation}
Applying the boundedness of $\varphi_u\in L^{\infty}(H^2)$, we conclude that
\begin{align}
\nonumber &\left|\intO \divergence(\varphi_u\v_k)\laplace\varphi_k\dx\right|
	= \left|\intO (\grad\varphi_u\cdot\v_k+ \varphi_u\divergence(\v_k))\laplace\varphi_k\dx\right|
	\leq C\normh{2}{\varphi_u}\normH{1}{\v_k}\norml{2}{\laplace\varphi_k}\\
&\quad \leq \delta_1\normH{1}{\v_k}^2 + C\normL{2}{\grad\varphi_k}\normL{2}{\grad\laplace\varphi_k}
	\label{lin_eq_4} \leq \delta_1\normH{1}{\v_k}^2 + \tfrac{m}{16}\normL{2}{\grad\laplace\varphi_k}^2 +C\normL{2}{\grad\varphi_k}^2,
\end{align}
with $\delta_1>0$ to be chosen later. Using the Sobolev embeddings $H^1\subset L^3,~\H^1\subset \L^6$ and $\H^2\subset \L^{\infty}$, we infer that
\begin{align}
\nonumber &\left|\intO \divergence(\varphi_k\v_u)\laplace\varphi_k\dx\right|
	%= \left|\intO (\grad\varphi_k\cdot\v_u + \varphi_k\divergence(\v_u))\laplace\varphi_k\dx\right|\\
	\leq C\left(\normL{2}{\grad\varphi_k}\normL{\infty}{\v_u} + \norml{3}{\varphi_k}\norml{6}{\divergence(\v_u)}\right)\norml{2}{\laplace\varphi_k}\\
&\quad\leq C\normH{2}{\v_u}\normh{1}{\varphi_k}\norml{2}{\laplace\varphi_k}
	\label{lin_eq_5} \leq C\normH{2}{\v_u}^2\normh{1}{\varphi_k}^2 + \tfrac{m}{16}\norml{2}{\laplace\varphi_k}^2.
\end{align}
Since $\h(\varphi_u),\h'(\varphi_u),\psi''(\varphi_u)\in L^{\infty}(\OmegaT),\,\sigma_u\in L^{\infty}(L^6)$ with bounded norm, we easily obtain that
\begin{align}
\nonumber &\left|\intO \big(-F_2 - (P\sigma_u-A)\h'(\varphi_u)\varphi_k - P\sigma_k\h(\varphi_u) +m(\psi''(\varphi_u)\varphi_k-\chi\sigma_k+F_3)\big)\laplace\varphi_k\dx\right|\\
\label{lin_eq_6}&\quad \leq C\left(\normh{1}{\varphi_k}^2 + \norml{2}{\sigma_k}^2 + \norml{2}{F_2}^2 + \norml{2}{F_3}^2\right) + \tfrac{m}{16}\norml{2}{\laplace\varphi_k}^2.
\end{align}
With similar arguments and using the Sobolev embedding $H^1\subset L^4$, it follows that
\begin{align}
\nonumber\left|\intO \big((P\sigma_u-A-u)\h'(\varphi_u)\varphi_k + P\sigma_k\h(\varphi_u) + F_2\big)\varphi_k\dx\right| &\leq C(1+\norml{2}{u}^2)\normh{1}{\varphi_k}^2 \\
\label{lin_eq_6a}&\quad + C\left(\norml{2}{\sigma_k}^2 + \norml{2}{F_2}^2\right).
\end{align}
Again using the boundedness of $\h'(\varphi_u)\in L^{\infty}(\OmegaT)$, the Sobolev embedding $H^1\subset L^6$ and Gagliardo-Nirenberg's inequality, we calculate
\begin{align*}
\left|\intO u\h'(\varphi_u)\varphi_k\laplace\varphi_k\dx\right|&\leq C\norml{2}{u}\norml{6}{\varphi_k}\norml{3}{\laplace\varphi_k}\\
	&\leq C\norml{2}{u}\norml{6}{\varphi_k}\norml{2}{\laplace\varphi_k}^{1/2}\left(\norml{2}{\laplace\varphi_k}^{\frac{1}{2}} + \normL{2}{\grad\laplace\varphi_k}^{1/2}\right)\\
& \leq C_{\delta_2,\delta_3}\norml{2}{u}^2\norml{6}{\varphi_k}^2 + \delta_2 \norml{2}{\laplace\varphi_k}^2 + 	\delta_3\norml{2}{\laplace\varphi_k}\normL{2}{\grad\laplace\varphi_k}\\
&\leq C_{\delta_2,\delta_3}\norml{2}{u}^2\norml{6}{\varphi_k}^2 + (\delta_2 + \delta_3)\norml{2}{\laplace\varphi_k}^2 +  \tfrac{\delta_3}{4}\normL{2}{\grad\laplace\varphi_k}^2,
\end{align*}
with $\delta_2,\delta_3>0$ arbitrary. Then, choosing $\delta_2,\delta_3$ sufficiently small, we conclude that
\begin{equation}
\label{lin_eq_7} \left|\intO u\h'(\varphi_u)\varphi_k\laplace\varphi_k\dx\right|\leq C\norml{2}{u}^2\norml{6}{\varphi_k}^2 + \tfrac{m}{16}\left( \norml{2}{\laplace\varphi_k}^2 +  \normL{2}{\grad\laplace\varphi_k}^2\right).
\end{equation}
Due to the Sobolev embeddings $H^1\subset L^p,~\H^1\subset \L^p,~p\in [1,6],$ and the boundedness of $\varphi_u\in L^{\infty}(H^1)$, we obtain
\begin{align}
\nonumber &\left|\intO \big(\divergence(\varphi_u\v_k)+\divergence(\varphi_k\v_u)\big)\varphi_k\dx\right| = \left|\intO \big(\grad\varphi_u\cdot\v_k+ \varphi_u\divergence(\v_k)+\grad\varphi_k\cdot\v_u + \varphi_k\divergence(\v_u)\big)\varphi_k\dx\right|\\
&\quad\leq C\left(\normh{1}{\varphi_u}\normH{1}{\v_k}\normh{1}{\varphi_k} + \normH{1}{\v_u}\normh{1}{\varphi_k}^2\right)
	\label{lin_eq_8}\leq \delta_4\normH{1}{\v_k}^2 + C_{\delta_4}(1+\normH{1}{\v_u}^2)\normh{1}{\varphi_k}^2,
\end{align}
with $\delta_4>0$ to be chosen later. Next, we apply the Sobolev embeddings $H^1\subset L^p,~\H^1\subset \L^p$ for $p\in [1,6]$, $H^2\subset L^{\infty}(\Omega)$ and the boundedness of $\varphi_u\in L^{\infty}(H^2)$ to get
\begin{align}
\nonumber &\left|\intO \big((\mu_u+\chi\sigma_u)\grad\varphi_k + (\mu_k+\chi\sigma_k)\grad\varphi_u+\mathbf{F}_2\big)\cdot(\v_k-\w_k)\right|\\
\nonumber &\quad\leq C\left(\normh{2}{\mu_u+\chi\sigma_u}\normh{1}{\varphi_k} + \norml{2}{\mu_k+\chi\sigma_k} + \normL{2}{\F}\right)\normH{1}{\v_k-\w_k}\\
\label{lin_eq_9}&\quad\leq C_{\delta_5}\left(\normh{2}{\mu_u+\chi\sigma_u}^2\normh{1}{\varphi_k}^2 + \norml{2}{\mu_k+\chi\sigma_k}^2 + \normL{2}{\F}^2+\normH{1}{\w_k}^2\right)  + \delta_5\normH{1}{\v_k}^2,
\end{align}
with $\delta_5>0$ to be chosen later. We recall that the $L^2$-orthogonal projection $\mathbb{P}_k$ onto $\mathcal{W}_k$ is continuous on $H^1$. Consequently, choosing $v= (b_i^k + \chi (\sigma_k{,}w_i)_{L^2})w_i$ in (\ref{approx_problem_eq_1b}), summing the resulting identities over $i=1,...,k$, using the boundedness of $\psi''(\varphi_u)\in L^{\infty}(\OmegaT)$ and applying (\ref{lin_eq_3a}), it follows that
\begin{align*}
\norml{2}{\mu_k+\chi\sigma_k}^2&\leq \norml{2}{\laplace\varphi_k}\norml{2}{\mu_k+\chi\sigma_k} + C(\norml{2}{\varphi_k}+\norml{2}{F_2})\norml{2}{\mu_k+\chi\sigma_k}\\
&\leq \normL{2}{\grad\varphi_k}^{1/2}\normL{2}{\grad\laplace\varphi_k}^{1/2}\norml{2}{\mu_k+\chi\sigma_k} + C(\norml{2}{\varphi_k}+\norml{2}{F_2})\norml{2}{\mu_k+\chi\sigma_k}\\
&\leq \left(\delta_6\normL{2}{\grad\laplace\varphi_k} + \tfrac{1}{4\delta_6}\normL{2}{\grad\varphi_k} + C\left(\norml{2}{\varphi_k}+\norml{2}{F_2}\right)\right)\norml{2}{\mu_k+\chi\sigma_k},
\end{align*}
meaning
\begin{equation}
\label{lin_eq_10} \norml{2}{\mu_k+\chi\sigma_k}\leq \left(\delta_6\normL{2}{\grad\laplace\varphi_k} + \tfrac{1}{4\delta_6}\normL{2}{\grad\varphi_k} + C\big(\norml{2}{\varphi_k}+\norml{2}{F_2}\big)\right)
\end{equation}
for $\delta_6>0$ arbitrary.
For the last term on the right-hand side of (\ref{lin_eq_1}), we obtain
\begin{equation}
\label{lin_eq_11}\left|\intO 2\eta\D\v_k\colon\grad\w_k + \nu\v_k\cdot\w_k\dx\right|\leq C\normH{1}{\w_k}^2 + \delta_7\normH{1}{\v_k}^2,
\end{equation}
for $\delta_7>0$ to be chosen. Plugging in (\ref{lin_eq_2})-(\ref{lin_eq_3}), (\ref{lin_eq_4})-(\ref{lin_eq_11}) into (\ref{lin_eq_1}), using Korn's inequality and chosing $\delta_i,~i\in \{1,4,5,6,7\}$ sufficiently small, we obtain that
\begin{align}
\nonumber&\ddt \normh{1}{\varphi_k}^2 + \norml{2}{\laplace\varphi_k}^2 + \normL{2}{\grad\laplace\varphi_k}^2 + \normH{1}{\v_k}^2\\
\label{lin_eq_12}&\quad\leq \beta(t)\normh{1}{\varphi_k(t)}^2+ C\Bigg(\normh{1}{\sigma_k}^2 + \normH{1}{\w_k}^2 + \normL{2}{\F}^2+\sum_{i=2}^{3}\norml{2}{F_i}^2 \Bigg),
\end{align}
where $\beta(t)\coloneqq C\big(1+\normH{2}{\v_u(t)}^2 +\normh{2}{\mu_u(t)+\chi\sigma_u(t)}^2+ \norml{2}{u(t)}^2 \big)$.
\\[1ex]
Due to the definition of $\UR$ and using Theorem \ref{THM:EXS}, it follows that $\beta\in L^1(0,T)$.
From the boundedness of $\h(\varphi_u),\h'(\varphi_u)\in L^{\infty}(\OmegaT),\,\sigma_u\in L^{\infty}(L^6)$ and due to (\ref{lin_eq_0}), we infer that
\begin{equation}
\label{lin_eq_13}\normH{1}{\w_k}\leq C\left(\norml{2}{\sigma_k} + \normh{1}{\varphi_k} + \norml{2}{F_1}\right).
\end{equation}
Multiplying \eqref{approx_problem_eq_1c} with $\sigma_k$, integrating by parts and using \eqref{approx_problem_eq_1cc}, the boundedness of $\h'(\varphi_u)\in L^{\infty}(\OmegaT),\,\sigma_u\in L^{\infty}(L^6)$ and the non-negativity of $\h(\cdot)$ yields
\begin{equation*}
\normL{2}{\grad\sigma_k}^2 + K\norm{\sigma_k}_{L^2(\del\Omega)}^2 = \left|\intO \big(F_4-\h'(\varphi_u)\varphi_k\sigma_u\big)\sigma_k\dx\right| \leq\delta_8\normh{1}{\sigma_k}^2 + C_{\delta_8}\left(\norml{2}{\varphi_k}^2 + \norml{2}{F_4}^2\right)
\end{equation*}
for $\delta_8>0$ arbitrary. Choosing $\delta_8$ sufficiently small and using Poincaré's inequality, this implies that
\begin{equation}
\label{lin_eq_14}\normh{1}{\sigma_k}\leq C(\norml{2}{\varphi_k} + \norml{2}{F_4}).
\end{equation}
Plugging in (\ref{lin_eq_13})-(\ref{lin_eq_14}) into (\ref{lin_eq_12}), we end up with
\begin{align}
\ddt \normh{1}{\varphi_k}^2 + \norml{2}{\laplace\varphi_k}^2 + \normL{2}{\grad\laplace\varphi_k}^2 + \normH{1}{\v_k}^2
\label{lin_eq_15}\leq \beta(t)\normh{1}{\varphi_k}^2 + C\left(\normL{2}{\grad F_3}^2 + \normL{2}{\F}^2 + \sum_{i=1}^{4}\norml{2}{F_i}^2\right).
\end{align}
Recalling \eqref{approx_problem_eq_initial} and using elliptic regularity theory, an application of Gronwall's Lemma to (\ref{lin_eq_15}) gives
\begin{equation}
\label{lin_eq_16} \norm{\varphi_k}_{L^{\infty}(H^1)\cap L^2(H^3)} + \norm{\v_k}_{L^2(\H^1)}\leq C\left(\norm{\grad F_3}_{L^2(\OmegaT)} + \norm{\F}_{L^2(\OmegaT;\R^3)} + \sum_{i=1}^{4}\norm{F_i}_{L^2(\OmegaT)}\right).
\end{equation}

\textit{Step 3:} Using (\ref{lin_eq_10}) and (\ref{lin_eq_14}), from (\ref{lin_eq_16}) we obtain
\begin{equation}
\label{lin_eq_17}\norm{\sigma_k}_{L^2(H^1)} + \norm{\mu_k}_{L^2(L^2)}\leq C\left(\norm{\grad F_3}_{L^2(\OmegaT)} + \norm{\F}_{L^2(\OmegaT;\R^3)} + \sum_{i=1}^{4}\norm{F_i}_{L^2(\OmegaT)}\right).
\end{equation}
Now, choosing $v=\lambda_i b_i^kw_i$ in (\ref{approx_problem_eq_1b}), summing the resulting identities over $i=1,...,k,$ and integrating by parts, we have
\begin{align*}
\normL{2}{\grad\mu_k}^2 &= \left|\intO \grad\big(-\laplace\varphi_k + \psi''(\varphi_u)\varphi_k -\chi\sigma_k + F_3\big)\cdot\grad\mu_k\dx\right|\\
&\leq \frac{1}{2}\left(\normL{2}{\grad\big(-\laplace\varphi_k + \psi''(\varphi_u)\varphi_k -\chi\sigma_k + F_3\big)}^2 + \normL{2}{\grad\mu_k}^2\right),
\end{align*}
which implies
\begin{equation*}
\normL{2}{\grad\mu_k}^2\leq \left(\normL{2}{\grad\big(-\laplace\varphi_k + \psi''(\varphi_u)\varphi_k -\chi\sigma_k + F_3\big)}^2\right).
\end{equation*}
Integrating this inequality in time from $0$ to $T$ and using (\ref{lin_eq_16})-(\ref{lin_eq_17}), we obtain
\begin{equation}
\label{lin_eq_18}\norm{\grad\mu_k}_{L^2(\L^2)}\leq C\left(\norm{\grad F_3}_{L^2(\OmegaT)} + \norm{\F}_{L^2(\OmegaT;\R^3)} + \sum_{i=1}^{4}\norm{F_i}_{L^2(\OmegaT)}\right).
\end{equation}
%Finally, using elliptic regularity theory and the non-negativity of $\h(\cdot)$, from \eqref{approx_problem_eq_1c}-\eqref{approx_problem_eq_1cc} we obtain
%\begin{equation*}
%\normh{2}{\sigma_k}^2\leq C\left(\norml{2}{F_4}^2+\normh{1}{\varphi_k}^2 + \norml{2}{\sigma_k}^2\right),
%\end{equation*}
%where we used the boundedness of $\h'(\varphi_u)\in L^{\infty}(\OmegaT),\,\sigma_u\in L^{\infty}(H^1)$. Integrating this inequality in time from $0$ to $T$ and using (\ref{lin_eq_16})-(\ref{lin_eq_17}), we conclude that
%\begin{equation}
%\label{lin_eq_19}\norm{\sigma_k}_{L^2(H^2)}\leq C\left(\norm{\grad F_3}_{L^2(\OmegaT)} + \norm{\F}_{L^2(\OmegaT;\R^3)} + \sum_{i=1}^{4}\norm{F_i}_{L^2(\OmegaT)}\right).
%\end{equation}

\textit{Step 4:} To get an estimate for the pressure, we test \eqref{approx_problem_eq_1d} with $\q_k\in \H^1$, where $\q_k$ satisfies
\begin{equation}
\divergence(\q_k) = p_k\quad\text{in }\Omega,\quad \q_k =\frac{|\Omega|}{|\Gamma|}(p_k)_\Omega\, \n\quad\text{on }\Gamma
%\q_k = \left(\frac{1}{|\Omega|}\intO p_k\dx\right)\n\quad\text{on }\del\Omega, 
\tand
\label{lin_eq_21}\normH{1}{\q_k}\leq C\norml{2}{p_k}.
\end{equation}
Therefore, using the boundedness of $\mu_u+\chi\sigma_u\in L^{\infty}(L^2),~\grad\varphi_u\in L^{\infty}(\H^1)$, we obtain that
\begin{align*}
\norml{2}{p_k}^2 &\leq C\left(\normH{1}{\v_k}^2 + \normL{3}{\grad\varphi_k}^2 + \normh{1}{\mu_k+\chi\sigma_k}^2 + \normL{2}{\F}^2\right).
\end{align*}
Integrating this inequality in time from $0$ to $T$ and using (\ref{lin_eq_16})-(\ref{lin_eq_18}), we get
\begin{equation}
\label{lin_eq_20}\norm{p_k}_{L^2(L^2)}\leq C\left(\norm{\grad F_3}_{L^2(\OmegaT)} + \norm{\F}_{L^2(\OmegaT;\R^3)} + \sum_{i=1}^{4}\norm{F_i}_{L^2(\OmegaT)}\right).
\end{equation}
Using a comparison argument in (\ref{WF:LIN3}) together with (\ref{lin_eq_16})-(\ref{lin_eq_18}), it follows that
\begin{equation}
\label{lin_eq_22}\norm{\varphi_k}_{H^1((H^1)^*)}\leq C\left(\norm{\grad F_3}_{L^2(\OmegaT)} + \norm{\F}_{L^2(\OmegaT;\R^3)} + \sum_{i=1}^{4}\norm{F_i}_{L^2(\OmegaT)}\right).
\end{equation}
Summarising (\ref{lin_eq_16})-(\ref{lin_eq_22}) gives
\begin{align}
\label{lin_eq_24}
\norm{(\varphi_k,\mu_k,\sigma_k,\v_k,p_k)}_{\V_2}  
 \leq C\left(\norm{\grad F_3}_{L^2(\OmegaT)} + \norm{\F}_{L^2(\OmegaT;\R^3)} + \sum_{i=1}^{4}\norm{F_i}_{L^2(\OmegaT)}\right).
\end{align}
\textit{Step 5:} Due to (\ref{lin_eq_24}), we can pass to the limit in (\ref{approx_problem_eq_1a})-(\ref{approx_problem_eq_1c}) and in the weak formulation of (\ref{approx_problem_eq_1d})-(\ref{approx_problem_eq_1f}) to deduce that (\ref{WF:LIN0})-(\ref{WF:LIN5}) and (\ref{EST:LIN}) hold. The initial condition is attained due to the compact embedding $H^1((H^1)^*)\cap L^{\infty}(H^1)\hookrightarrow C(L^2)$ (see \cite[sect.\,8, Cor.\,4]{Simon}). Moreover, the estimate \eqref{EST:LIN} results from the weak-star lower semicontinuity of the $\V_2$-norm. Recall that, according to Lemma \ref{LEM:EST}, any weak solution of (LIN) satisfies the inequality \eqref{EST:LIN}. Hence, uniqueness of the weak solution immediately follows due to linearity of the system.
\end{proof}

\subsection{Fréchet differentiability}

Now, this result can be used to prove Fréchet differentiability of the control-to-state operator:
\begin{lemma}
	\label{LEM:FRE}
	The control-to-state operator $S$ is Fréchet differentiable on $\UR$, i.e. for any $u\in\UR$ there exists a unique bounded linear operator 
	\begin{align*}
		S'(u): L^2(L^2) \to \V_2, \quad h \mapsto S'(u)[h]=\big(\varphi'_u,\mu'_u,\v'_u,\sigma'_u,p'_u\big)[h],
	\end{align*}
	where $\V_2$ is the space that was introduced in Definition \ref{DEF:CSO}, such that
	\begin{align*}
		\frac{\norm{S(u+h)-S(u)-S'(u)[h]}_{\V_2}}{\norm{h}_{L^2(L^2)}} \to 0 \qquad\text{as}\; \norm{h}_{L^2(L^2)}\to 0.
	\end{align*}
	Moreover, for any $u\in\U$ and $h\in L^2(L^2)$, the Fréchet derivative $\big(\varphi'_u,\mu'_u,\v'_u,\sigma'_u,p'_u\big)[h]$ is the unique weak solution of the system \textnormal{(LIN)} with 
	\begin{align*}
		F_1,F_3,F_4=0,\quad \F=\mathbf{0} \tand F_2= -h\,\h(\varphi_u).
	\end{align*}
\end{lemma}

\begin{proof}
	Let $C$ denote a generic nonnegative constant that depends only on $R$, $\Omega$, $\Gamma$ and $T$ and may change its value from line to line. To prove Fréchet differentiability we must consider the difference 
	\begin{align*}
		(\varphi,\mu,\sigma,\v,p):=(\varphi_{u+h},\mu_{u+h},\v_{u+h},\sigma_{u+h},p_{u+h})-(\varphi_u,\mu_u,\sigma_u,\v_u,p_u)
	\end{align*}
	for some arbitrary $u\in\UR$ and $h\in L^2(L^2)$ with $u+h\in\UR$. Therefore, we assume that $\norm{h}_{L^2(L^2)}<\delta$ for some sufficiently small $\delta>0$. Now, we Taylor expand the nonlinear terms in (CHB) to pick out the linear contributions. We obtain that
	\begin{align*}
		\h(\varphi_{u+h}) - \h(\varphi_{u}) &= \h'(\varphi_u)\varphi + \CR_1, \\
		\sigma_{u+h}\h(\varphi_{u+h}) - \sigma_{u}\h(\varphi_{u}) &= \sigma\h(\varphi_{u}) + \sigma_{u}\h'(\varphi_u)\varphi + \CR_2, \\
		(u+h)\h(\varphi_{u+h}) - u\h(\varphi_{u}) &= u\h'(\varphi_u)\varphi + h\h(\varphi_u) + \CR_3,\\
		(\mu_{u+h}+\chi\sigma_{u+h})\grad\varphi_{u+h} - (\mu_{u}+\chi\sigma_{u})\grad\varphi_{u} &= (\mu_{u}+\chi\sigma_{u})\grad\varphi + (\mu+\chi\sigma)\grad\varphi_{u} + \CR_4, \\
		\divergence(\varphi_{u+h}\v_{u+h}) - \divergence(\varphi_{u}\v_{u}) &= \divergence(\varphi\v_{u}) + \divergence(\varphi_{u}\v) +\CR_5,\\
		\psi'(\varphi_{u+h}) - \psi'(\varphi_{u}) &= \psi''(\varphi_{u})\varphi + \CR_6,
	\end{align*}
	where the nonlinear remainders are given by
	\begin{align*}
		\CR_1&:= \tfrac 1 2 \h''(\zeta)(\varphi_{u+h} - \varphi_{u})^2,	\\
		\CR_2&:= (\sigma_{u+h}-\sigma_u)(\h(\varphi_{u+h}) - \h(\varphi_{u})) + \tfrac 1 2 \sigma_u\h''(\zeta)(\varphi_{u+h} - \varphi_{u})^2,\\
		\CR_3&:= \tfrac 1 2 u\,\h''(\zeta)(\varphi_{u+h} - \varphi_{u})^2 + h\big(\h(\varphi_{u+h})-\h(\varphi_{u})\big),\\
		\CR_4&:= \big[(\mu_{u+h} - \mu_{u}) +\chi(\sigma_{u+h}-\sigma_{u})\big]
		(\grad\varphi_{u+h} - \grad\varphi_{u}),\\	
		\CR_5&:= \divergence\big[(\varphi_{u+h} - \varphi_{u})(\v_{u+h} - \v_{u})\big],\\
		\CR_6&:= \tfrac 1 2 \psi'''(\xi) (\varphi_{u+h} - \varphi_{u})^2	
	\end{align*}
	with $\zeta=\theta_1\varphi_{u+h}+(1-\theta_1)\varphi_{u}$ and $\xi=\theta_2\varphi_{u+h}+(1-\theta_2)\varphi_{u}$ for some $\theta_1,\theta_2\in[0,1]$. This means that the difference $(\varphi,\mu,\sigma,\v,p)$ is the weak solution of (LIN) with
	\begin{gather*}
		F_1 = P\CR_2 - A\CR_1,\quad \F = \CR_4,\quad F_2 =  P\CR_2 - A\CR_1 - \CR_3-\CR_5 -h\,\h(\varphi_u),\quad F_3 = \CR_6,\quad F_4=-\CR_2.
	\end{gather*}
	By a simple computation, one can show that these functions have the desired regularity. Now, we write $(\varphi_u^h,\mu_u^h,\v_u^h,\sigma_u^h,p_u^h)$ to denote the weak solution of (LIN) with 
	\begin{align*}
		F_1,F_3,F_4=0,\quad \F=\mathbf{0} \tand F_2= -h\,\h(\varphi_u).
	\end{align*}
	and $(\varphi_{\CR}^h,\mu_{\CR}^h,\v_{\CR}^h,\sigma_{\CR}^h,p_{\CR}^h)$ to denote the weak solution of (LIN) with 
	\begin{align}
	\label{EQ:F}
		F_1 = P\CR_2 - A\CR_1,\;\; \F = \CR_4,\;\;\, F_2 =  P\CR_2 - A\CR_1 - \CR_3-\CR_5,\;\;\, F_3 = \CR_6,\;\;\, F_4=-\CR_2.
	\end{align}
	Because of linearity of the system (LIN) and uniqueness of its solution, it follows that
	\begin{align*}
		(\varphi_{u+h},\mu_{u+h},\v_{u+h},\sigma_{u+h},p_{u+h}) - (\varphi_u,\mu_u,\sigma_u,\v_u,p_u) - (\varphi_u^h,\mu_u^h,\v_u^h,\sigma_u^h,p_u^h)
			=(\varphi_{\CR}^h,\mu_{\CR}^h,\v_{\CR}^h,\sigma_{\CR}^h,p_{\CR}^h).
	\end{align*}
	We conclude from Theorem \ref{THM:EXS} that $\zeta$ and $\xi$ are uniformly bounded. This yields
	\begin{align*}
		\norm{\psi'''(\zeta)}_{L^\infty(\OmegaT)}\le C,\quad \norm{\psi''''(\zeta)}_{L^\infty(\OmegaT)}\le C \tand \norm{\h''(\zeta)}_{L^\infty(\OmegaT)}\le C.
	\end{align*} 
	Moreover, since $\h(\cdot)$ is Lipschitz continuous, inequality \eqref{IEQ:L2} implies that
	\begin{align*}
		\norm{\h(\varphi_{u+h})-\h(\varphi_{u})}_{L^\infty(\OmegaT)} \le C\,\norm{\varphi_{u+h}-\varphi_{u}}_{L^\infty(\OmegaT)} \le C\,\norm{h}_{L^2(L^2)}^{1/2}.
	\end{align*} 
	Together with the Lipschitz estimates from Lemma \ref{LEM:LIP} we obtain that
	\begin{align*}
		\norm{\CR_i}_{L^2(L^2)} \le C\, \norm{h}_{L^2(L^2)}^{3/2},\quad i\in\{1,2,3,6\}.
	\end{align*} 
	Since $\norm{\grad\varphi_{u+h}-\grad\varphi_{u}}_{L^\infty(L^6)} \le \norm{\varphi_{u+h}-\varphi_{u}}_{L^\infty(H^2)} \le 2C_1$ we have
	\begin{align}
	\label{EST:R51}
		&\norm{(\grad\varphi_{u+h}-\grad\varphi_{u})(\v_{u+h}-\v_{u})}_{L^2(L^2)}^2 \notag\\
		&\quad \le \norm{\grad\varphi_{u+h}-\grad\varphi_{u}}_{L^\infty(L^6)}\, \norm{\grad\varphi_{u+h}-\grad\varphi_{u}}_{L^\infty(L^2)}\, \norm{\v_{u+h}-\v_{u}}_{L^2(L^6)}^2 \notag\\
		&\quad \le C\, \norm{\varphi_{u+h}-\varphi_{u}}_{L^\infty(H^1)}\, \norm{\v_{u+h}-\v_{u}}_{L^2(H^1)}^2
	\end{align} 
	and thus
	\begin{align*}
	%\label{EST:R52}
		\norm{\CR_5}_{L^2(L^2)} 
		&\le C\,\norm{\varphi_{u+h}-\varphi_{u}}_{L^\infty(\OmegaT)}\, \norm{\v_{u+h}-\v_{u}}_{L^2(H^1)} 
			+ C\,\norm{(\grad\varphi_{u+h}-\grad\varphi_{u})(\v_{u+h}-\v_{u})}_{L^2(L^2)} \notag\\
		&\le C\,\norm{h}_{L^2(L^2)}^{3/2}.
	\end{align*} 
	Since $u\mapsto \mu_u$ and $u\mapsto \sigma_u$ are also Lipschitz continuous with respect to the norm on $L^2(H^1)$, we can proceed similarly to \eqref{EST:R51} to bound the term $\CR_4$. This yields
	\begin{align*}
	\norm{\CR_4}_{L^2(L^2)} \le C\,\norm{h}_{L^2(L^2)}^{3/2}.
	\end{align*} 
	Furthermore, $\grad\CR_6$ can be bounded by
	\begin{align*}
		\norm{\grad\CR_6}_{L^2(L^2)} 
		&\le  C\, \norm{\grad\xi}_{L^\infty(L^6)} \, \norm{\varphi_{u+h}-\varphi_u}_{L^\infty(L^6)}^2 
			+ C\, \norm{\grad\varphi_{u+h}-\grad\varphi_u}_{L^2(L^3)} \, \norm{\varphi_{u+h}-\varphi_u}_{L^\infty(L^6)} \\
		&\le  C\, \norm{\grad\xi}_{L^\infty(\OmegaT)} \, \norm{\varphi_{u+h}-\varphi_u}_{L^\infty(H^1)}^2 
			+ C\, \norm{\varphi_{u+h}-\varphi_u}_{L^2(H^2)} \, \norm{\varphi_{u+h}-\varphi_u}_{L^\infty(H^1)} \\
		&\le C\, \norm{h}_{L^2(L^2)}^{2}.
	\end{align*} 
	This finally implies that
	\begin{gather*}
		\norm{F_i}_{L^2(L^2)} \le C\, \norm{h}_{L^2(L^2)}^{3/2}\;\;\text{for}\; i=1,2,3,4, \quad
		\norm{\F}_{L^2(L^2)} \le C\, \norm{h}_{L^2(\L^2)}^{3/2}, \quad
		\norm{\grad F_3}_{L^2(L^2)} \le C\, \norm{h}_{L^2(L^2)}^{3/2}
	\end{gather*}
	where $F_i$ denote the functions given by \eqref{EQ:F}. Hence, due to estimate \eqref{EST:LIN} we obtain that
	\begin{align*}
		\norm{(\varphi_{\CR}^h,\mu_{\CR}^h,\v_{\CR}^h,\sigma_{\CR}^h,p_{\CR}^h)}_{\V_2} \le C\, \norm{h}_{L^2(L^2)}^{3/2},
	\end{align*}
	which completes the proof.
\end{proof}

\section{The optimal control problem}

In this section we analyze the optimal control problem that was presented in the introduction: We intend to minimize the \textit{\bfseries cost functional}
\begin{align*}
I(\varphi,\mu,\sigma,\v,p,u)&:= 
\frac{\upalpha_0} 2 \norml{2}{\varphi(T)-\varphi_f}^2
+ \frac{\upalpha_1} 2 \norm{\varphi-\varphi_d}_{L^2(\Omega_T)}^2
+ \frac{\upalpha_2} 2\norm{\mu-\mu_d}_{L^2(\Omega_T)}^2 \\
&\qquad + \frac{\upalpha_3} 2\norm{\sigma-\sigma_d}_{L^2(\Omega_T)}^2
+ \frac{\upalpha_4} 2\norm{\v-\v_d}_{L^2(\Omega_T)}^2
+ \frac{\kappa} 2\norm{u}_{L^2(\Omega_T)}^2
\end{align*}
subject to the following conditions
\begin{itemize}
	\item $u$ is an admissible control, i.e., $u\in\U$,
	\item $(\varphi,\mu,\sigma,\v,p)$ is a strong solution of the system (CHB) to the control $u$.
\end{itemize}
Using the control-to-state operator we can formulate this optimal control problem alternatively as
\begin{align}
	\label{OCP}
		\text{Minimize} &\quad J(u) \quad \text{s.t.}\quad u\in\U,
\end{align}
where the \textit{\bfseries reduced cost functional} $J$ is defined by
\begin{align}
\label{DEF:J}
	J(u):=I(S(u),u)=I(\varphi_u,\mu_u,\sigma_u,\v_u,p_u,u), \quad u\in\U.
\end{align}
A globally/locally optimal control of this optimal control problem is defined as follows:

\bigskip

\begin{definition}
	Let $u\in\U$ be any admissible control.
	\begin{itemize}
		\itema $\u$ is called a \textbf{(globally) optimal control} of the problem \eqref{OCP} if $J(\u)\le J(u)$ for all $u\in\U$.
		\itemb $\u$ is called a \textbf{locally optimal control} of the problem \eqref{OCP} if there exists some $\delta>0$ such that  $J(\u)\le J(u)$ for all $u\in\U$ with $\norm{u-\u}_{L^2(\Omega_T)}<\delta$.
	\end{itemize}
	In this case, $(\varphi_\u,\mu_\u,\sigma_\u,\v_\u,p_\u)$ is called the corresponding \textbf{globally/locally optimal state}.
\end{definition}

\smallskip

\subsection{Existence of a globally optimal control}

Of course, the optimal control problem \eqref{OCP} does only make sense if there is at least one globaly optimal solution. This is established by the following Theorem. The proof is rather short as most of the work has already been done in the previous chapter.

\begin{theorem}
	The optimization problem \eqref{OCP} possesses a globally optimal solution.
\end{theorem}
\begin{proof}
	This result can be proved by the direct method of calculus of variations: Obviously, the functional $J$ is bounded from below by zero. Therefore the infimum $m:=\inf_{u\in\U}J(u)$ exists and we can find a minimizing sequence $(u_k)\subset\U$ with $J(u_k)\to m$ as $k\to\infty$. As the set $\U$ is weakly sequentially compact, there exists $\u\in\U$ such that $u_k\wto\u$ in $L^2(L^2)$ after extraction of a subsequence. Now, according to Lemma \ref{LEM:COMP} we obtain that 
	\begin{alignat*}{6}
		&\varphi_{u_k} &&\wto \varphi_\u\quad &&\text{in}\; H^1(L^2)\cap L^2(H^4),\quad
		&&\varphi_{u_k} &&\to \varphi_\u &&\text{in}\; C(W^{1,r})\cap C(\overline{\OmegaT}),~r\in [1,6),\\
		&\mu_{u_k}&&\wto\mu_\u\quad &&\text{in}\; L^2(H^2),
		&&\v_{u_k}&&\wto\v_\u\quad &&\text{in}\; L^2(H^2),\\
		&\sigma_{u_k}&&\wto\sigma_\u\quad &&\text{in}\; L^2(H^2), \quad  
		&&p_{u_k}&&\wto p_\u\quad &&\text{in}\; L^2(H^1).
	\end{alignat*}
	after another subsequence extraction (in particular, it follows that $\varphi_{u_k}(T)\wto \varphi_\u(T)$ in $L^2$). Furthermore, Lemma~\ref{LEM:COMP} yields that
	\begin{equation*}
	S(\u) = (\varphi_\u,\mu_\u,\sigma_\u,\v_\u,p_\u),
	\end{equation*}
	hence $(\u,S(\u))$ is an admissible control-state pair.
	From the weak lower semicontinuity of the cost functional $J$ we can conclude that
	\begin{align*}
	J(\u) \le \underset{k\to\infty}{\lim\inf} J(u_k) = \underset{k\to\infty}{\lim} J(u_k) = m
	\end{align*}
	and $J(\u)=m$ immediately follows by the definition of the infimum. This means that $\u$ is a globally optimal control with corresponding state $S(\u) = (\varphi_\u,\mu_\u,\sigma_\u,\v_u,p_\u)$.
\end{proof}
%\begin{proof}
%	This result can be proved by the direct method of calculus of variations: Obviously, the functional $J$ is bounded from below by zero. Therefore the infimum $m:=\inf_{u\in\U}J(u)$ exists and we can find a minimizing sequence $(u_k)\subset\U$ with $J(u_k)\to m$ as $k\to\infty$. As the set $\U$ is weakly sequentially compact, there exists $\u\in\U$ such that $u_k\wto\u$ after extraction of a subsequence. Now, according to Lemma \ref{LEM:COMP} we obtain that $\varphi_{u_k}\wto \varphi_u$ in $H^1(L^2)$ and $\mu_{u_k}\wto \mu_u$, $\v_{u_k}\wto \v_u$, $\sigma_{u_k}\wto \sigma_u$ in $L^2(\Omega_T)$ after another subsequence extraction. By the fundamental theorem of calculus this also implies that $\varphi_{u_k}(T)\wto \varphi_u(T)$ in $L^2$. From the weak lower semicontinuity of the cost functional $J$ we can conclude that
%	\begin{align*}
%		J(\u) \le \underset{k\to\infty}{\lim\inf} J(u_k) = \underset{k\to\infty}{\lim} J(u_k) = m
%	\end{align*}
%	and $J(\u)=m$ immediately follows by the definition of infimum. This means that $\u$ is a globally optimal solution.
%\end{proof}

Of course this theorem does not provide uniqueness of the globally optimal solution $\u$. As the control-to-state operator is nonlinear we cannot expect the cost functional to be convex. Therefore, it is possible that the optimization problem has several locally optimal solutions or even several globally optimal solutions. In the following, since numerical methods will (in general) only detect local minimizers, our goal is to characterize locally optimal solutions by a necessary optimality condition.
\\[1ex]
As the control-to-state operator is Fréchet differentiable according to Lemma \ref{LEM:FRE}, Fréchet differentiability of the cost functional easily follows by chain rule. If $\u\in\U$ is a locally optimal solution, it must hold that $J'(\u)[u-\u]\ge 0$ for all $u\in \U$. The problem is, that (up to now) the operator $J'(\u)$ cannot be described explicitly. However, we will show that $J'(\u)$ can be described by means of the so-called adjoint state which will be introduced in the next subsection.

\subsection{The adjoint system}

The following system of partial differential equations is referred to as the \textit{\bfseries adjoint system}:
\begin{subequations}
	\begin{empheq}[left=\textnormal{(ADJ)}\hspace{-3pt}\empheqlbrace]{align}
	\label{EQ:ADJ1}
	\divergence(\w) &= 0 &&\text{in}\;\OmegaT, \\
	\label{EQ:ADJ2}
	-\eta\laplace\w + \nu\w &=  -\grad q +\varphi_u\grad\vartheta + \upalpha_4(\v_u-\v_d)  &&\text{in}\;\OmegaT,\\
	\label{EQ:ADJ3}
	\delt\vartheta+\v_u\cdot\grad\vartheta &= -(P\sigma_u-A-u)\h'(\varphi_u)\vartheta - \h'(\varphi_u)\sigma_u \rho - \psi''(\varphi_u)\tau + \laplace\tau\notag\\
	& + \grad(\mu_u+\chi\sigma_u)\cdot\w+(P\sigma_u-A)\h'(\varphi_u)q - \upalpha_1(\varphi_u\sminus \varphi_d) &&\text{in}\;\OmegaT,\\
	\label{EQ:ADJ4}
	\tau &= \grad\varphi_u\cdot\w + m\laplace\vartheta + \upalpha_2(\mu_u-\mu_d) &&\text{in}\;\OmegaT,\\
	\label{EQ:ADJ5}
	\laplace\rho - \h(\varphi_u)\rho &= - \chi\tau + P\h(\varphi_u)\vartheta +\chi\grad\varphi_u\scdot\w - P\h(\varphi_u)q + \upalpha_3(\sigma_u \sminus \sigma_d)
	&&\text{in}\;\OmegaT,\\[2mm]
	\label{EQ:ADJ6}
	\deln\rho &= -K\rho &&\text{on}\;\GammaT,\\
	\label{EQ:ADJ7}
	\deln\vartheta &=0 &&\text{on}\;\GammaT,\\
	\label{EQ:ADJ8}
	0 &= (2\eta \D\w - q\I + \varphi_u\vartheta\I)\n &&\text{on}\;\GammaT,\\
	\label{EQ:ADJ9}
	\deln\tau &= \vartheta\v_u\cdot\n - (\mu_u+\chi\sigma_u)\w\cdot\n &&\text{on}\;\GammaT,\\[2mm]
	\label{EQ:ADJ10}
	\vartheta(T)&=\upalpha_0(\varphi_u(T)-\varphi_f) &&\text{in}\; \Omega.
	\end{empheq}
\end{subequations}

\bigskip
A weak solution of this system of equations is defined as follows:
\begin{definition}
	Let $u\in\UR$ be any control and let $(\varphi_u,\mu_u,\sigma_u,\v_u,p_u)$ denote its corresponding state. Then the quintuple $(\vartheta,\tau,\rho,\w,q)$ is called a weak solution of the adjoint system if it lies in $\V_2$ and satisfies the equations
	\begin{alignat}{2}
	\label{WF:ADJ0} 
	\vartheta(T) &= \upalpha_0(\varphi_u(T)-\varphi_f) &&\quad\text{a.e. in}\; \Omega,\\
	\label{WF:ADJ1} 
	\divergence(\w) &= 0 &&\quad\text{a.e. in}\; \OmegaT 
	\end{alignat}
	and
	\begin{align}
	\label{WF:ADJ2} 
	&\intO 2\eta\D\w\scolon\grad\tw + \nu\w\scdot\tw - q\divergence(\tw) \dx = \sminus\intO \vartheta\grad\varphi_u\scdot\tw + \vartheta\varphi_u \divergence(\tw) \dx 
	+ \intO \upalpha_4(\v_u \sminus \v_d)\scdot\tw \dx, \\
	\label{WF:ADJ3} 
	&\big\langle \delt\vartheta, \tphi \big\rangle_{H^1} = -\intO  \big[ (P\sigma_u-A-u)\h'(\varphi_u)\vartheta  + \h'(\varphi_u)\sigma_u\rho + \psi''(\varphi_u)\tau \notag\\[-2mm]
	& \hspace{120pt} - (P\sigma_u-A)\h'(\varphi_u) q - \vartheta\divergence(\v_u) + \upalpha_1 (\varphi_u-\varphi_d)\big] \tphi \dx \notag\\
	& \hspace{90pt}+ \intO \big[\vartheta\v_u -(\mu_u+\chi\sigma_u)\w -\grad\tau \big]\cdot\grad\tphi\dx, \\
	\label{WF:ADJ4} 
	&\intO \tau\ttau \dx = \intO \big[ \grad\varphi_u\cdot\w + \upalpha_2(\mu_u-\mu_d) \big]\ttau - m\grad\vartheta\cdot\grad\ttau \dx, \\
	\nonumber
	&-\intO \grad\rho\cdot\grad\trho \dx - K \intG \rho\trho \dS = \intO \big[ -\chi\tau+P\h(\varphi_u)\vartheta + \chi\grad\varphi_u\cdot\w \\[-2mm]
	\label{WF:ADJ5} & \hspace{180pt} - P\h(\varphi_u)q + \upalpha_3(\sigma_u-\sigma_d) + \h(\varphi_u)\rho \big]\trho\dx,
	\end{align}
	for a.e. $t\in (0,T)$ and all $\tphi,\ttau,\trho \in H^1,~\tw\in \H^1$.
\end{definition}

Next, we will show that the adjoint system is uniquely solvable.

\begin{theorem}
	\label{LEM:ADJ}
	Let $u\in\U$ be arbitrary. Then, the adjoint system \textnormal{(ADJ)} has a unique weak solution $(\vartheta,\tau,\rho,\w,q)$. In addition, it holds that
	\begin{equation*}
	\rho\in L^2(H^2),\quad \w \in  L^2(\H^2),\quad q \in L^2(H^1)
	\end{equation*}
	and
	\begin{align}
		\nonumber &\norm{(\vartheta,\tau,\rho,\w,q)}_{\V_2} +  \norm{\rho}_{L^2(H^2)} + \norm{\w}_{L^2(\H^2)} +  \norm{q}_{L^2(H^1)}\\
		\label{LEM:ADJ:EST}&\quad\leq C\norm{(\varphi_u-\varphi_d,\varphi_u(T)-\varphi_f,\mu_u-\mu_d,\sigma_u-\sigma_d,\v_u-\v_d)}_{\V_3}.
		\end{align}
	This unique solution is called the \textbf{adjoint state} or \textbf{costate}.
\end{theorem}

\begin{proof} We will only show a-priori-estimates for the solutions of the adjoint system. The justification can be carried out rigorously within a Galerkin scheme as in the proof of Proposition \ref{PROP:LIN}. In particular, equations \eqref{WF:ADJ0}-\eqref{WF:ADJ5} are satisfied by the Galerkin solutions with the duality product replaced by the $L^2$-scalar-product and $\varphi_{u}(T)-\varphi_f$ replaced by $\mathbb{P}_k(\varphi_{u}(T)-\varphi_f)$, where $\mathbb{P}_k$ denotes the $L^2$-orthogonal projection onto the k-dimensional subspaces spanned by the eigenfunctions of the Neumann-Laplace operator (see proof of Proposition \ref{PROP:LIN}). In the following, we will suppress the subscript $k$. 
\\[1ex]	
Hölder's and Young's inequalities will be frequently used as well as a generic constant $C$ which does not depend on the approximating solutions deduced within the Galerkin scheme. The approach will be split into several steps. 
\\[1ex]
\textit{Step 1:} We define $\pi\coloneqq q-\varphi_{u}\vartheta$. Then, from \eqref{EQ:ADJ1}-\eqref{EQ:ADJ2}, \eqref{EQ:ADJ8}, we see that $(\w,\pi)$ is for almost every $t\in (0,T)$ a solution of
\begin{alignat*}{3}
-\eta\laplace\w + \nu\w + \grad\pi &= \f&&\quad\text{a.e. in }\Omega,\\
\divergence(\w) &= 0&&\quad\text{a.e. in }\Omega,\\
(2\eta\D\w - \pi\I)\n &= \mathbf{0}&&\quad\text{a.e. on }\del\Omega,
\end{alignat*}
where $\f \coloneqq -\vartheta\nabla\varphi_{u} + \upalpha_4(\v_{u}-\v_d)$. Applying \cite[Theorem 3.2]{ShibataShimizu}, we obtain (for a.e. $t\in (0,T)$)
\begin{equation}
\label{adj_eq_1}\normH{2}{\w} + \normh{1}{\pi}\leq C\norml{2}{\f}.
\end{equation}
In particular, by the definition of $\pi$ and $\f$ and using that 
\begin{equation*}
\normL{2}{\vartheta\nabla\varphi_u}\leq C\normh{1}{\vartheta\varphi_u},
\end{equation*}
we have 
\begin{equation}
\label{adj_eq_2}\normH{2}{\w} + \normh{1}{q}\leq C\left(\normh{1}{\vartheta\varphi_u} + \upalpha_4\normL{2}{\v_{u}-\v_d}\right).
\end{equation}
Hence, we have to estimate the first term on the r.h.s. of this equation. Using the boundedness of $\varphi_u\in L^{\infty}(H^2)\cap L^{\infty}(\OmegaT)$ and the Sobolev embedding $H^1\subset L^3(\Omega)$, we obtain
\begin{align*}
\nonumber \normh{1}{\varphi_u\vartheta} &\leq C\left(\norml{2}{\varphi_u\vartheta} + \normL{2}{\vartheta\grad\varphi_u}+\normL{2}{\varphi_u\grad\vartheta}\right)\\
&\leq C\left(\norml{\infty}{\varphi_u}\norml{2}{\vartheta} + \normL{6}{\grad\varphi_u}\norml{3}{\vartheta} + \norml{\infty}{\varphi_u}\normL{2}{\grad\vartheta}\right)\\
&\leq C\normh{1}{\vartheta}.
\end{align*}
Plugging in this inequality into \eqref{adj_eq_2}, we infer that
\begin{align}
\label{adj_eq_5} & \normH{2}{\w} +  \normh{1}{q} \leq C\left( \normh{1}{\vartheta} + \normL{2}{\v_u-\v_d}\right).
\end{align}
\textit{Step 2:} Choosing $\ttau = \chi\rho$ in \eqref{WF:ADJ4}, $\trho = -\rho$ in \eqref{WF:ADJ5} and summing the resulting identities, we obtain
\begin{align*}
\intO |\grad\rho|^2\dx + K\int_{\del\Omega}|\rho|^2\dS + \intO \h(\varphi_u)|\rho|^2\d x&= -\intO \big(P\h(\varphi_u)\vartheta - P\h(\varphi_u)q + \upalpha_3(\sigma_u-\sigma_d)\big)\rho\dx\\
&\quad + \chi \intO \upalpha_2(\mu_u-\mu_d)\rho - m\grad\vartheta\cdot\grad\rho\dx.
\end{align*}
Using the boundedness of $\h(\varphi_u)\in L^{\infty}(\OmegaT)$ and the non-negativity of $\h(\cdot)$, (\ref{adj_eq_5}) and Poincaré's inequality, this implies that
\begin{equation}
\label{adj_eq_6}\normh{1}{\rho}\leq C\left( \normh{1}{\vartheta} + \norml{2}{\mu_u-\mu_d} + \norml{2}{\sigma_u-\sigma_d} + \normL{2}{\v_u-\v_d}\right).
\end{equation}
Choosing $\ttau = \tau$ in \eqref{WF:ADJ4}, integrating by parts, using the boundedness of $\grad\varphi_u\in L^{\infty}(\L^3)$ and (\ref{adj_eq_5}), we obtain
\begin{align*}
\norml{2}{\tau}^2 &= \left|\intO \big(\grad\varphi_u\cdot\w + m\laplace\vartheta + \upalpha_2(\mu_u-\mu_d)\big)\tau\dx\right|\\
&\leq C\left(\normL{3}{\grad\varphi_u}\normL{6}{\w} + \norml{2}{\laplace\vartheta} + \norml{2}{\mu_u-\mu_d}\right)\norml{2}{\tau}\\
&\leq C\left( \norml{2}{\vartheta} + \normL{2}{\v_u-\v_d}+\norml{2}{\laplace\vartheta}+\norml{2}{\mu_u-\mu_d}\right)\norml{2}{\tau}.
\end{align*}
Consequently, we have
\begin{equation}
\label{adj_eq_6a}\norml{2}{\tau}\leq C\left( \norml{2}{\vartheta} +\norml{2}{\laplace\vartheta}+ \normL{2}{\v_u-\v_d}+\norml{2}{\mu_u-\mu_d}\right).
\end{equation}
\textit{Step 3:} Choosing $\tphi = \laplace\vartheta - \vartheta$ in \eqref{WF:ADJ3}, $\ttau = \laplace^2\vartheta - \laplace\vartheta$ in \eqref{WF:ADJ4}, integrating by parts and summing the resulting identities, we obtain
\begin{align}
\nonumber &-\frac{1}{2}\ddt \left(\norml{2}{\vartheta}^2+\normL{2}{\grad\vartheta}^2\right) + m\left(\norml{2}{\laplace\vartheta}^2 + \normL{2}{\grad\laplace\vartheta}^2\right) \\
\nonumber&= \intO \big((P\sigma_u-A)\h'(\varphi_u)(\vartheta-q) + \h'(\varphi_u)\sigma_u\rho + \psi''(\varphi_u)\tau +\upalpha_1(\varphi_u-\varphi_d)\big)\big(\vartheta-\laplace\vartheta\big)\dx\\
\nonumber &\quad + \intO \divergence(\v_u)\vartheta(\laplace\vartheta-\vartheta) + \big(\vartheta\v_u - (\mu_u+\chi\sigma_u)\w\big)\cdot(\grad\laplace\vartheta-\grad\vartheta) + u\h'(\varphi_u)\vartheta(\laplace\vartheta-\vartheta)\dx\\
\label{adj_eq_7}&\quad -\intO \big(\grad(\grad\varphi_u\cdot\w) +  \upalpha_2\grad(\mu_u-\mu_d)\big)\cdot\grad\laplace\vartheta + \big(\grad\varphi_u\cdot\w + \upalpha_2(\mu_u-\mu_d)\big)\laplace\vartheta\dx.
\end{align}
Using the boundedness of $\h'(\varphi_u),\psi''(\varphi_u)\in L^{\infty}(\OmegaT),\,\sigma_u\in L^{\infty}(L^6)$ and (\ref{adj_eq_5})-(\ref{adj_eq_6}), we calculate
\begin{align}
\nonumber &\left|\intO \big(-A\h'(\varphi_u)(\vartheta-q) + \h'(\varphi_u)\sigma_u\rho + \upalpha_1(\varphi_u-\varphi_d)\big)\big(\vartheta-\laplace\vartheta\big)\dx\right|\\
\nonumber&\leq C\left(\norml{2}{\vartheta}^2 + \norml{2}{\rho}^2 + \norml{2}{q}^2 + \norml{2}{\varphi_u-\varphi_d}^2\right) + \tfrac{m}{16}\norml{2}{\laplace\vartheta}^2\\
\label{adj_eq_8}&\leq C\left(\normh{1}{\vartheta}^2 + \norml{2}{\varphi_u-\varphi_d}^2 + \norml{2}{\sigma_u-\sigma_d}^2 + \norml{2}{\mu_u-\mu_d}^2 + \normL{2}{\v_u-\v_d}^2\right)  + \tfrac{m}{16}\norml{2}{\laplace\vartheta}^2.
\end{align}
%	By the Gagliardo-Nirenberg inequality, we obtain
%	\begin{equation*}
%	\norml{3}{\laplace\vartheta}\leq C\norml{2}{\laplace\vartheta}^{\frac{1}{2}}\normh{1}{\laplace\vartheta}^{\frac{1}{2}}\leq C\left(\norml{2}{\laplace\vartheta}+\norml{2}{\laplace\vartheta}^{\frac{1}{2}}\normL{2}{\nabla\laplace\vartheta}^{\frac{1}{2}}\right)\leq C\left(\norml{2}{\laplace\vartheta}+\normL{2}{\nabla\laplace\vartheta}\right).
%	\end{equation*}
Hence, using the boundedness of $\h'(\varphi_u)\in L^{\infty}(\Omega_T),\,\sigma_{u}\in L^{\infty}(Q)$ and (\ref{adj_eq_5}) yields
\begin{align}
\nonumber\left|\intO P\sigma_u\h'(\varphi_u)(\vartheta-q)(\vartheta-\laplace\vartheta)\dx \right|	&\leq C(\norml{2}{\vartheta}+\norml{2}{q})\left(\norml{2}{\vartheta}+\norml{2}{\laplace\vartheta}\right)\\
\nonumber&\leq C\left(\normh{1}{\vartheta}^2 + \norml{2}{q}^2\right) + \tfrac{m}{16}\norml{2}{\laplace\vartheta}^2\\
\label{adj_eq_8a}&\leq C\left(\normh{1}{\vartheta}^2 + \normL{2}{\v_u - \v_d}^2\right)+ \tfrac{m}{16}\norml{2}{\laplace\vartheta}^2.
\end{align} 
Furthermore, using the boundedness of $\psi''(\varphi_u)\in L^{\infty}(\OmegaT)$, (\ref{adj_eq_6a}) and the inequality
\begin{equation*}
\norml{2}{\laplace\vartheta}^2\leq \normL{2}{\grad\vartheta}\normL{2}{\grad\laplace\vartheta}\quad\forall \vartheta\in H^3,~\deln\vartheta = 0\quad\text{a.e. on }\del\Omega,
\end{equation*}
we obtain
\begin{align}
\nonumber\left|\intO \psi''(\varphi_u)\tau(\vartheta-\laplace\vartheta)\dx\right| &\leq C(\norml{2}{\vartheta}^2 + \normL{2}{\v_u-\v_d}^2 + \norml{2}{\mu_u-\mu_d}^2) + C\norml{2}{\laplace\vartheta}^2\\
\nonumber&\leq C\left(\norml{2}{\vartheta}^2 + \normL{2}{\v_u-\v_d}^2 + \norml{2}{\mu_u-\mu_d}^2\right) + C\normL{2}{\grad\vartheta}\normL{2}{\grad\laplace\vartheta}\\
\label{adj_eq_8b}&\leq  C\left(\normh{1}{\vartheta}^2 + \normL{2}{\v_u-\v_d}^2 + \norml{2}{\mu_u-\mu_d}^2\right) + \tfrac{m}{16}\normL{2}{\grad\laplace\vartheta}^2.
\end{align}
From the Sobolev embeddings $H^1\subset L^3(\Omega),~\H^1\subset \L^6$, we obtain
\begin{align}
\nonumber \left|\intO\divergence(\v_u)\vartheta(\laplace\vartheta-\vartheta)\dx\right|&\leq \norml{6}{\divergence(\v_u)}\norml{3}{\vartheta}(\norml{2}{\vartheta}+\norml{2}{\laplace\vartheta})\\
\label{adj_eq_9}&\leq C\left(1+\normH{2}{\v_u}^2\right)\normh{1}{\vartheta}^2 + \tfrac{m}{16}\norml{2}{\laplace\vartheta}^2.
\end{align}
Using the Sobolev embeddings $H^1\subset L^6,\, \H^1\subset \L^3, \H^2\subset \L^{\infty}$, (\ref{adj_eq_5}) and the boundedness of $\mu_u +\chi\sigma_u\in L^{\infty}(L^2)$, we calculate
\begin{align}
\nonumber &\left|\intO \big(\vartheta\v_u - (\mu_u+\chi\sigma_u)\w\big)\scdot(\grad\laplace\vartheta-\grad\vartheta)\dx\right|
\leq \left(\norml{6}{\vartheta}\normL{3}{\v_u} + \norml{2}{\mu_u+\chi\sigma_u}\normL{\infty}{\w}\right)\left(\normL{2}{\grad\vartheta} + \normL{2}{\grad\laplace\vartheta}\right)\\
&\quad\leq C\left(1+\normH{1}{\v_u}^2 + \normL{2}{\v_u-\v_d}^2 \right)\normh{1}{\vartheta}^2
\label{adj_eq_10}  + C\normL{2}{\v_u-\v_d}^2 + \tfrac{m}{16}\normL{2}{\grad\laplace\vartheta}^2.
\end{align}
With similar arguments and using the boundedness of $\grad\varphi_u\in L^{\infty}(\L^6)$, we obtain
\begin{align}
\nonumber&\left|\intO \upalpha_2\grad(\mu_u-\mu_d)\cdot\grad\laplace\vartheta + \big(\grad\varphi_u\cdot\w + \upalpha_2(\mu_u-\mu_d)\big)\laplace\vartheta\dx\right|\\
\nonumber &\quad\leq \upalpha_2\normL{2}{\grad(\mu_u-\mu_d)}\normL{2}{\grad\laplace\vartheta} + \left(\normL{6}{\grad\varphi_u}\normL{3}{\w} + \upalpha_2\norml{2}{\mu_u-\mu_d}\right)\norml{2}{\laplace\vartheta}\\
\label{adj_eq_11}&\quad\leq C\left(\normh{1}{\vartheta}^2+\normh{1}{\mu_u-\mu_d}^2 + \normL{2}{\v_u-\v_d}^2\right) + \tfrac{m}{16}\left(\norml{2}{\laplace\vartheta}^2 + \normL{2}{\grad\laplace\vartheta}^2\right).
\end{align}
Now, with exactly the same arguments as used for (\ref{lin_eq_7}), we get
\begin{equation}
\label{adj_eq_12}\left|\intO u\h'(\varphi_u)\vartheta(\laplace\vartheta-\vartheta)\dx\right|\leq C(1+\norml{2}{u}^2)\normh{1}{\vartheta}^2 + \tfrac{m}{16}\left(\norml{2}{\laplace\vartheta}^2 + \normL{2}{\grad\laplace\vartheta}^2\right).
\end{equation}
It remains to analyse the term
\begin{equation*}
\intO \grad(\grad\varphi_u\cdot\w)\cdot\grad\laplace\vartheta\dx = \intO (\grad^2\varphi_u\w)\cdot\grad\laplace\vartheta + (\grad\w^T\grad\varphi_u)\cdot\grad\laplace\vartheta.
\end{equation*}
For the first term, we apply the Sobolev embedding $\H^2\subset \L^{\infty}$, the boundedness of $\varphi_u\in L^{\infty}(H^2)$ and (\ref{adj_eq_5}) to obtain
\begin{align}
\nonumber\left|\intO(\grad^2\varphi_u\w)\cdot\grad\laplace\vartheta\dx\right|&\leq \normL{2}{\grad^2\varphi_u}\normL{\infty}{\w}\normL{2}{\grad\laplace\vartheta}\\
\nonumber&\leq C\left(\normh{1}{\vartheta} + \normL{2}{\v_u-\v_d}\right)\normL{2}{\grad\laplace\vartheta}\\
\label{adj_eq_13} &\leq C\left(\normh{1}{\vartheta}^2 + \normL{2}{\v_u-\v_d}^2\right)+ \tfrac{m}{16}\normL{2}{\grad\laplace\vartheta}^2.
\end{align}
With similar arguments and using the Sobolev embeddings $\H^1\subset \L^6, \H^1\subset \L^3$, we infer that
\begin{align}
\nonumber\left|\intO (\grad\w^T\grad\varphi_u)\cdot\grad\laplace\vartheta\dx\right| &\leq \normL{6}{\grad\varphi_u}\normH{2}{\w}\normL{2}{\grad\laplace\vartheta}\\
\label{adj_eq_14}&\leq C\left(\normh{1}{\vartheta}^2 + \normL{2}{\v_u-\v_d}^2\right) + \tfrac{m}{16}\normL{2}{\grad\laplace\vartheta}^2.
\end{align}
Collecting (\ref{adj_eq_8})-(\ref{adj_eq_14}) and plugging in into (\ref{adj_eq_7}), we obtain
\begin{align}
\label{adj_eq_14a} -\frac{1}{2}\ddt \left(\norml{2}{\vartheta}^2+\normL{2}{\grad\vartheta}^2\right) + \frac{m}{2}\left(\norml{2}{\laplace\vartheta}^2 + \normL{2}{\grad\laplace\vartheta}^2\right)\leq  \beta_1(t)\normh{1}{\vartheta(t)}^2 + \beta_2(t)
\end{align}
where 
\begin{align*}
\beta_1(t) &\coloneqq C\left(1+\normH{2}{\v_u(t)}^2 + \normh{1}{(\mu_u+\chi\sigma_u)(t)}^2+\normL{2}{(\v_u-\v_d)(t)}^2 + \norml{2}{u(t)}^2 + \normh{4}{\varphi_u(t)}^2\right), \\
\beta_2(t)&\coloneqq C\left(\norml{2}{\varphi_u-\varphi_d}^2 + \normh{1}{\mu_u-\mu_d}^2 + \norml{2}{\sigma_u-\sigma_d}^2 + \normL{2}{\v_u-\v_d}^2\right).
\end{align*}
Due to Theorem \ref{THM:EXS} and \textnormal{(A5)}, it is easy to check that $\beta_1,~\beta_2\in L^1(0,T)$. Therefore, integrating (\ref{adj_eq_14a}) in time from $s\in (0,T)$ to $T$ and using that $\varphi_u\in C(H^1)$ with bounded norm, a Gronwall argument yields
\begin{align}
\label{adj_eq_15}\norm{\vartheta}_{L^{\infty}(H^1)\cap L^2(H^3)}\leq C\norm{(\varphi_u-\varphi_d,\varphi_u(T)-\varphi_f,\mu_u-\mu_d,\sigma_u-\sigma_d,\v_u-\v_d)}_{\V_3}.
\end{align}
Together with (\ref{adj_eq_5})-(\ref{adj_eq_6a}), this implies
\begin{align}
\nonumber&\norm{\tau}_{L^2(\OmegaT)} + \norm{\rho}_{L^2(H^1)} + \norm{\w}_{L^2(\H^2)}  + \norm{q}_{L^2(H^1)}\\
\label{adj_eq_16}&\quad\leq C\norm{(\varphi_u-\varphi_d,\varphi_u(T)-\varphi_f,\mu_u-\mu_d,\sigma_u-\sigma_d,\v_u-\v_d)}_{\V_3}.
\end{align}
%	Since the terms on the right-hand side of (\ref{adj_eq_4}) belong to $L^{8/3}(\H^1)$, upon integrating in time from $0$ to $T$ and using we obtain
%	\begin{equation}
%	\label{adj_eq_17}\norm{\w}_{L^{8/3}(\H^1)} +  \norm{q}_{L^{8/3}(L^2)}\leq C\norm{(\varphi_u-\varphi_d,\varphi_u(T)-\varphi_f,\mu_u-\mu_d,\sigma_u-\sigma_d,\v_u-\v_d)}_{\V_3}.
%	\end{equation}
We now take $\ttau = -\laplace\tau$ in \eqref{WF:ADJ4} and integrate by parts to get
\begin{equation}
\label{adj_eq_18}\normL{2}{\grad\tau}^2 = \intO \big((\grad^2\varphi_u\w) + (\grad\w^T\grad\varphi_u) + m\grad\laplace\vartheta +\upalpha_2\grad(\mu_u-\mu_d)\big)\cdot\grad\tau\dx.
\end{equation}
For the last two terms on the right-hand side of this identity, we easily obtain
\begin{equation}
\label{adj_eq_19}\left|\intO m\grad\laplace\vartheta +\upalpha_2\grad(\mu_u-\mu_d)\big)\cdot\grad\tau\dx\right|\leq C\left(\normL{2}{\grad\laplace\vartheta}^2 + \normh{1}{\mu_u-\mu_d}^2\right) + \frac{1}{4}\normL{2}{\grad\tau}^2.
\end{equation}
For the other terms, we use the Sobolev embeddings $\H^1\subset \L^6,\H^1\subset \L^3,\,\H^2\subset \L^{\infty},$ and the boundedness of $\varphi_u\in L^{\infty}(H^2)$ to deduce
\begin{equation}
\label{adj_eq_20}\left|\intO \big((\grad^2\varphi_u\w) + (\grad\w^T\grad\varphi_u) \big)\cdot\grad\tau\dx\right|\leq C\normH{2}{\w}^2 + \frac{1}{4}\normL{2}{\grad\tau}^2.
\end{equation}
Plugging in (\ref{adj_eq_19})-(\ref{adj_eq_20}), we obtain
\begin{equation*}
\normL{2}{\grad\tau}^2\leq C\left(\normL{2}{\grad\laplace\vartheta}^2 + \normh{1}{\mu_u-\mu_d}^2 \normH{2}{\w}^2\right).
\end{equation*}
Integrating this inequality in time from $0$ to $T$, using the boundedness of $\varphi_u\in L^2(H^4)$ and (\ref{adj_eq_15}), (\ref{adj_eq_16}), we infer that
\begin{align}
\label{adj_eq_21}\norm{\grad\tau}_{L^{2}(\L^2)} \leq C\norm{(\varphi_u-\varphi_d,\varphi_u(T)-\varphi_f,\mu_u-\mu_d,\sigma_u-\sigma_d,\v_u-\v_d)}_{\V_3}.
\end{align}
Summarizing (\ref{adj_eq_15})-(\ref{adj_eq_16}) and (\ref{adj_eq_21}) and using a comparison argument in (\ref{EQ:ADJ3}), we deduce that
\begin{align}
\nonumber &\norm{(\vartheta,\tau,\rho,\w,q)}_{\V_2} + \norm{\w}_{L^2(\H^2)} +  \norm{q}_{L^2(H^1)}\\
\label{adj_eq_22}&\quad\leq C\norm{(\varphi_u-\varphi_d,\varphi_u(T)-\varphi_f,\mu_u-\mu_d,\sigma_u-\sigma_d,\v_u-\v_d)}_{\V_3}.
\end{align}
\textit{Step 5:} Since (\ref{WF:ADJ5}) is the weak formulation of (\ref{EQ:ADJ5})-(\ref{EQ:ADJ6}), by elliptic regularity theory we obtain
\begin{equation}
\label{adj_eq_31}\normh{2}{\rho}\leq C\norml{2}{\h(\varphi_u)\rho - \chi\tau + P\h(\varphi_u)\vartheta +\chi\grad\varphi_u\scdot\w - P\h(\varphi_u)q + \upalpha_4(\sigma_u \sminus \sigma_d)}.
\end{equation}
Using the boundedness of $\grad\varphi_u\in L^{\infty}(\L^3)$ and the Sobolev embedding $\H^1\subset \L^6$ we calculate
\begin{equation}
\label{adj_eq_32}\norml{2}{\grad\varphi_u\cdot\w}\leq C\normL{3}{\grad\varphi_u}\normL{6}{\w}\leq C\normH{1}{\w}.
\end{equation}
Therefore, using (\ref{adj_eq_22}) and (\ref{adj_eq_32}) in (\ref{adj_eq_31}), the boundedness of $\h(\varphi_u)\in L^{\infty}(\OmegaT)$ yields
\begin{align}
\label{adj_eq_33} \norm{\rho}_{L^2(H^2)}\leq C\norm{(\varphi_u-\varphi_d,\varphi_u(T)-\varphi_f,\mu_u-\mu_d,\sigma_u-\sigma_d,\v_u-\v_d)}_{\V_3}.
\end{align} 
Summarizing (\ref{adj_eq_22})-(\ref{adj_eq_33}), we have shown that
\begin{align}
\nonumber &\norm{(\vartheta,\tau,\rho,\w,q)}_{\V_2} +  \norm{\rho}_{L^2(H^2)} + \norm{\w}_{L^2(\H^2)} +  \norm{q}_{L^2(H^1)}\\
\label{adj_eq_34}&\quad\leq C\norm{(\varphi_u-\varphi_d,\varphi_u(T)-\varphi_f,\mu_u-\mu_d,\sigma_u-\sigma_d,\v_u-\v_d)}_{\V_3}.
\end{align}
\textit{Step 6:}
Because of (\ref{adj_eq_34}), we can pass to the limit in (\ref{WF:ADJ0})-(\ref{WF:ADJ5}) to obtain the existence of weak solutions. In particular, we infer that (\ref{EQ:ADJ2}), (\ref{EQ:ADJ4})-(\ref{EQ:ADJ8}) are fulfilled almost everywhere in the respective sets. We notice that (\ref{WF:ADJ0}) is fulfilled due to the compact embedding $H^1(H^{1\ast})\cap L^{\infty}(H^1)\subset C(L^2)$, see \cite[sect.\,8, Cor.~4]{Simon}. Moreover, the estimate \eqref{LEM:ADJ:EST} results from the weak(-star) lower semicontinuity of norms. Assuming that there exists any further weak solution of (ADJ) one can show (similarly to the above procedure) that it also satisfies inequality \eqref{adj_eq_34}. Hence, uniqueness of the weak solution follows due to linearity of the system.
\end{proof}

\subsection{Necessary conditions for local optimality}

In the following we characterize locally optimal solutions of \eqref{OCP} by necessary conditions which are particularly important for computational optimization. The adjoint variables can be used to express the variational inequality in a very concise form:

\begin{theorem}
	Let $\u\in\U$ be a locally optimal control of the minimization problem \eqref{OCP}. Then $\u$ satisfies the \textbf{variational inequality} that is
	\begin{align}
	\label{VIQ}
		\intOT \big[\kappa\u -\vartheta_\u\, \h(\varphi_\u) \big] (u-\u) \dxt \ge 0\quad\text{for all}\; u\in\U.
	\end{align}
\end{theorem}

\begin{proof}
	In Lemma \ref{LEM:FRE} we have showed that the control-to-state operator is Fréchet differentiable with respect to the norm on $\V_2$. Fréchet differentiability of the reduced cost functional $J$ immediately follows. Its derivative can be computed by chain rule. Hence, since $\u$ is a locally optimal control, it holds that
	\begin{align*}
		J\big(\u + t(u-\u)\big) \ge J(\u)
	\end{align*}
	for all $u\in\U$ and $t\in (0,1]$ sufficiently small. As the cost functional $J$ is Fréchet differentiable, we infer that
	\begin{align}
	\label{VIQ2}
		0 &\le \underset{t\searrow 0}{\lim}\; \frac 1 t \big[ J\big(\u + t(u-\u)\big) - J(\u)\big]= J'(\u)[u-\u] \notag\\
		&= \upalpha_0 \intO (\varphi_\u(T)-\varphi_f) \varphi'_\u[u-\u](T) \dx 
			+ \upalpha_1 \intOT (\varphi_\u-\varphi_d) \varphi'_\u[u-\u]\dx \notag\\
		&\quad + \upalpha_2 \intOT (\mu_\u-\mu_d) \mu'_\u[u-\u]\dx
			+ \upalpha_3 \intOT (\sigma_\u-\sigma_d) \sigma'_\u[u-\u]\dx  \notag\\
		&\quad + \upalpha_4 \intOT (\v_\u-\v_d)\cdot \v'_\u[u-\u]\dx
			+ \intOT \kappa \u(u-\u) \dx.
	\end{align}
	Therefore, it remains to show that the sum of the first five addends on the right-hand side of \eqref{VIQ2} is equal to $\intOT \vartheta_\u\h(\varphi_\u)(u-\u) \dx$. For brevity and to reduce the amount of indices we write $h:=u-\u$ and
	\begin{gather*}
		(\varphi,\mu,\sigma,\v,p) := (\varphi_\u,\mu_\u,\sigma_\u,\v_\u,p_\u), \qquad
		(\vartheta,\tau,\rho,\w,q) := (\vartheta_\u,\tau_\u,\rho_\u,\w_\u,q_\u), \\
		(\tvarphi,\tmu,\tsigma,\tv,\tp) := (\varphi_\u'[u-\u],\mu_\u'[u-\u],\sigma_\u'[u-\u],\v_\u'[u-\u],p_\u'[u-\u]).
	\end{gather*}
	In the following, the strategy is to test the weak formulations of the linearized system (which produces the Fréchet derivative) with the adjoint variables. Testing \eqref{EQ:LIN3} where $F_2 = -h\,\h(\varphi)$ with $\vartheta$ yields
	\begin{align}
		\label{EQ:VIQ1}
		0 &= \int_0^T \langle \delt\tvarphi, \vartheta\rangle_{H^1} \dt  + \intOT m\grad\tmu\cdot\grad\vartheta \dxt \notag \\
		&\quad + \intOT \big[\divergence(\varphi\tv) + \divergence(\tvarphi\v) - (P\sigma-A-u)\h'(\varphi)\tvarphi + h\,\h(\varphi)- P\h(\varphi)\tsigma\big]\,\vartheta\dxt.	
	\end{align}
	Since both $\tvarphi$ and $\vartheta$ lie in $H^1\big((H^1)^\ast\big)\cap L^2(H^1)$ integration by parts with respect to $t$ is permitted. We obtain
	\begin{align*}
		\int_0^T \langle \delt\tvarphi, \vartheta\rangle_{H^1} \dt = \upalpha_0 \intO \tvarphi(T)\big(\varphi(T)-\varphi_f\big)\dx 
			- \int_0^T \langle \delt\vartheta{,} \tvarphi\rangle_{H^1} \dt
	\end{align*}
	because of the initial value condition $\tvarphi(0) = 0$ and the final value condition $\vartheta(T)=\varphi(T)-\varphi_f$ which are satisfied almost everywhere in $\Omega$. The term $\delt\vartheta$ can be replaced using the weak formulation \eqref{WF:ADJ3} tested with $\tvarphi$. We obtain that
	\begin{align}
	\label{EQ:VIQ1.1}
	0 &= \upalpha_0 \intO \tvarphi(T)\big(\varphi(T)-\varphi_f\big)\dx + \intOT \grad\tau\cdot\grad\tvarphi + (\mu+\chi\sigma)\grad\tvarphi\cdot\w \dxt \notag\\
	&\quad +\intOT \h'(\varphi)\sigma\rho\tvarphi + \psi''(\varphi)\tau\tvarphi - (P\sigma-A)\h'(\varphi)q\tvarphi +\upalpha_1(\varphi-\varphi_d)\tvarphi \dxt\notag \\
	&\quad +\intOT \vartheta\,\divergence(\varphi\tv) - P\h(\varphi)\tsigma\vartheta + h\h(\varphi)\vartheta + m\grad\tmu\cdot\grad\vartheta  \dxt .
	\end{align}
	Since $\divergence(\w)=0$ almost everywhere in $\OmegaT$, we have $\T(\tv,\tp):\grad \w = 2\eta\D\tv:\grad \w = 2\eta\D\w:\grad \tv$. This equality, the weak formulation \eqref{WF:LIN2} tested with $\w$ and the weak formulation \eqref{WF:ADJ2} tested with $\tv$ can be used to deduce that
	\begin{align}
	\label{EQ:VIQ2}
	0 &= \intOT \T(\tv,\tp):\grad \w + \nu \tv\w - (\mu+\chi\sigma)\grad\tvarphi\cdot\w - (\tmu+\chi\tsigma)\grad\varphi\cdot\w \dxt \notag\\
	& \quad = \intOT \big[q\,\divergence(\tv) - \vartheta\grad\varphi\cdot\tv - \varphi\vartheta\,\divergence(\tv) + \upalpha_4(\v-\v_d)\cdot\tv \notag\\[-2mm]
	&\qquad\qquad\qquad - (\mu+\chi\sigma)\grad\tvarphi\cdot\w - (\tmu+\chi\tsigma)\grad\varphi\cdot\w \big] \dxt.
	\end{align}
	Proceeding similarly with the remaining linearized equations and adjoint variables gives
	\begin{align}
	\label{EQ:VIQ3}
	0 &= \intOT \tmu\tau - \psi''(\varphi)\tvarphi\tau + \chi\tsigma\tau - \grad\tvarphi\cdot\grad\tau \dxt \notag\\
	&=\intOT \grad\varphi\cdot\w \tmu -m\grad\vartheta\cdot\grad\tmu - \psi''(\varphi)\tvarphi\tau + \chi\tsigma\tau - \grad\tvarphi\cdot\grad\tau + \upalpha_2 (\mu-\mu_d)\tmu \dxt,\\
	\label{EQ:VIQ4}
	0 &= \intOT -q\,\divergence(\tv) + P\tsigma\h(\varphi)q + (P\sigma-A)\h'(\varphi)\tvarphi q \dxt, \\
	\label{EQ:VIQ5}
	0 &= -\intOT \grad\tsigma\cdot\grad\rho + \h'(\varphi)\tvarphi\sigma\rho + \h(\varphi)\tsigma\rho \dxt - \intGT K\tsigma\rho \dS\dt \notag\\
	& = \intOT P\h(\varphi)\vartheta\tsigma + \chi\grad\varphi\cdot\w\tsigma - P\h(\varphi)q\tsigma + \upalpha_3(\sigma-\sigma_d)\tsigma-\chi\tau\tsigma - \h'(\varphi)\tvarphi\sigma\rho \dxt.
	\end{align}
	Adding up the equations \eqref{EQ:VIQ1.1}-\eqref{EQ:VIQ5} we ascertain that a large number of terms cancels out. We obtain
	\begin{align}
	\label{EQ:VIQ6}
	0 &= \upalpha_0 \intO \big(\varphi(T)-\varphi_f\big)\tilde\varphi(T) \dx + \upalpha_1 \intOT (\varphi-\varphi_d)\varphi \dx + \upalpha_2 \intOT (\mu-\mu_d)\tmu \dxt  \notag\\
	&\quad + \upalpha_3 \intOT (\sigma-\sigma_d)\tsigma\dxt + \upalpha_4 \intOT (\v-\v_d)\tv \dxt + \intOT \h(\varphi)\vartheta h \dxt.
	\end{align}
	Together with \eqref{VIQ2} this completes the proof.
\end{proof}

As our set of admissible controls is a box-restricted subset of $L^2(L^2)$, a locally optimal control $\u$ can be characterized by a projection of $\frac 1 \kappa\, \vartheta_\u\, \h(\varphi_\u)$ onto the set $\U$. The following corollary is a standard result of optimal control theory:

\begin{corollary}
	Let $\u\in\U$ be a locally optimal control of the minimization problem \eqref{OCP}. Then $u$ is given implicitly by the \textbf{projection formula}
	\begin{align}
		\label{PRF}
		\u(x,t) = \mathbb P_{[a(x,t),b(x,t)]}\left(\frac 1 \kappa \,\vartheta_\u(x,t)\, \h\big(\varphi_\u(x,t)\big)\right)
	\end{align}
	for almost all $(x,t)\in\Omega_T$ where the projection $\mathbb P$ is defined by
	\begin{align*}
		\mathbb P_{[a,b]}(s) = \max\big\{a,\min\{b,s\}\big\}
	\end{align*}
		for any $a,b,s\in\R$ with $a\le b$. This constitutes another necessary condition for local optimality that is equivalent to condition \eqref{VIQ}.
\end{corollary}

\section*{Acknowledgements}

Matthias Ebenbeck was supported by the RTG 2339 ``Interfaces, Complex Structures, and Singular Limits" of the
German Science Foundation (DFG). The support is gratefully acknowledged.

%%%%%%%%%%%%%%%%%%%%%%%%%%%%%%%%%%%%%%%%%%%%%%%%%%%%%%%%%%%%%%%%%%%%%%%%%%%%%%%%%%%%%%%%%%%%%%%%%%%%%%%%%%%%%%%%%%%%
%*******************************************************************************************************************
% END OF DOCUMENT
%*******************************************************************************************************************
%%%%%%%%%%%%%%%%%%%%%%%%%%%%%%%%%%%%%%%%%%%%%%%%%%%%%%%%%%%%%%%%%%%%%%%%%%%%%%%%%%%%%%%%%%%%%%%%%%%%%%%%%%%%%%%%%%%%

\bigskip

%%%%%%%%%%%%%%%%%%%%%%%%%%%%%%%%%%%%%%%%%%%%%%%%%%%%%%%%%%%%%%%%%%%%%%%%%%%%%%%%%%%%%%%%%%%%%%%%%%%%%%%%%%%%%%%%%%%%
%*******************************************************************************************************************
% REFERENCES
%*******************************************************************************************************************
%%%%%%%%%%%%%%%%%%%%%%%%%%%%%%%%%%%%%%%%%%%%%%%%%%%%%%%%%%%%%%%%%%%%%%%%%%%%%%%%%%%%%%%%%%%%%%%%%%%%%%%%%%%%%%%%%%%%

\footnotesize
%\nocite{*} % Show all bib entries - both cited and uncited; comment this line to view only cited bib entries;
\bibliographystyle{plain}
\bibliography{oc-chb}

\begin{thebibliography}{10}

\bibitem{AbelsTerasawa}
H.~Abels and Y.~Terasawa.
\newblock On {S}tokes operators with variable viscosity in bounded and
  unbounded domains.
\newblock {\em Math. Ann.}, 344(2):381--429, 2009.

\bibitem{AgostiEtAl}
A.~Agosti, C.~Cattaneo, C.~Giverso, D.~Ambrosi, and P.~Ciarletta.
\newblock A computational framework for the personalized clinical treatment of
  glioblastoma multiforme.
\newblock {\em ZAMM - Journal of Applied Mathematics and Mechanics /
  Zeitschrift für Angewandte Mathematik und Mechanik}, 2018.

\bibitem{amann}
H.~Amann.
\newblock {\em {Linear and Quasilinear Parabolic Problems, Volume I: Abstract
  Linear Theory}}.
\newblock Birkhäuser Basel, 1995.

\bibitem{AstaninPreziosi}
S.~Astanin and L.~Preziosi.
\newblock Multiphase models of tumour growth.
\newblock In {\em Selected topics in cancer modeling}, Model. Simul. Sci. Eng.
  Technol., pages 223--253. Birkh\"{a}user Boston, Boston, MA, 2008.

\bibitem{BearerEtAl}
E.L. Bearer, J.S. Lowengrub, H.B. Frieboes, Y.L. Chuang, F.~Jin, S.M. Wise,
  M.~Ferrari, and V.~Agus, D.B.and~Cristini.
\newblock Multiparameter computational modeling of tumor invasion.
\newblock {\em Cancer Research}, 69(10):4493--4501, 2009.

\bibitem{Benosman}
C.~Benosman, B.~A{\"i}nseba, and A.~Ducrot.
\newblock {Optimization of Cytostatic Leukemia Therapy in an
  Advection--Reaction--Diffusion Model}.
\newblock {\em Journal of Optimization Theory and Applications},
  167(1):296--325, 2015.

\bibitem{Biswas}
T.~Biswas, S.~Dharmatti, and M.T. Mohan.
\newblock {Pontryagin's maximum principle for optimal control of the nonlocal
  Cahn-Hilliard-Navier-Stokes systems in two dimensions}.
\newblock {\em ArXiv e-prints:
  \href{https://arxiv.org/abs/1802.08413}{arXiv:1802.08413}}, 2018.

\bibitem{Cavaterra}
C.~Cavaterra, E.~Rocca, and H.~Wu.
\newblock {Long-Time Dynamics and Optimal Control of a Diffuse Interface Model
  for Tumor Growth}.
\newblock {\em Appl. Math. Optim.,
  \url{https://doi.org/10.1007/s00245-019-09562-5}}, 2019.

\bibitem{CoddingtonLevinson}
E.~A. Coddington and N.~Levinson.
\newblock {\em Theory of ordinary differential equations}.
\newblock McGraw-Hill Book Company, Inc., New York-Toronto-London, 1955.

\bibitem{ColliFarshbaf}
P.~Colli, M.H. Farshbaf-Shaker, G.~Gilardi, and J.~Sprekels.
\newblock Optimal boundary control of a viscous {C}ahn-{H}illiard system with
  dynamic boundary condition and double obstacle potentials.
\newblock {\em SIAM J. Control Optim.}, 53(4):2696--2721, 2015.

\bibitem{ColliGilardiHilhorst}
P.~Colli, G.~Gilardi, and D.~Hilhorst.
\newblock On a {C}ahn-{H}illiard type phase field system related to tumor
  growth.
\newblock {\em Discrete Contin. Dyn. Syst.}, 35(6):2423--2442, 2015.

\bibitem{ColliGilardiRoccaSprekels}
P.~Colli, G.~Gilardi, E.~Rocca, and J.~Sprekels.
\newblock Optimal distributed control of a diffuse interface model of tumor
  growth.
\newblock {\em Nonlinearity}, 30(6):2518--2546, 2017.

\bibitem{ColliGilardiSprekels}
P.~Colli, G.~Gilardi, and J.~Sprekels.
\newblock Optimal velocity control of a viscous {C}ahn-{H}illiard system with
  convection and dynamic boundary conditions.
\newblock {\em SIAM J. Control Optim.}, 56(3):1665--1691, 2018.

\bibitem{ChristiniLiLowengrubWise}
V.~Cristini, X.~Li, J.~S. Lowengrub, and S.~M. Wise.
\newblock Nonlinear simulations of solid tumor growth using a mixture model:
  invasion and branching.
\newblock {\em J. Math. Biol.}, 58(4-5):723--763, 2009.

\bibitem{EbenbeckGarcke2}
M.~Ebenbeck and H.~Garcke.
\newblock {On a Cahn-Hilliard-Brinkman model for tumour growth and its singular
  limits}.
\newblock {\em SIAM J. Math. Anal.}, 51(3):1868--1912.

\bibitem{EbenbeckGarcke}
M.~Ebenbeck and H.~Garcke.
\newblock {Analysis of a Cahn–Hilliard–Brinkman model for tumour growth
  with chemotaxis}.
\newblock {\em J.~Differential~Equations}, 266(9):5998--6036, 2019.

\bibitem{FrieboesEtAl}
H.B. Frieboes, J.~Lowengrub, S.~Wise, X.~Zheng, P.~Macklin, E.~Bearer, and
  V.~Cristini.
\newblock Computer simulation of glioma growth and morphology.
\newblock {\em NeuroImage}, 37 Suppl 1:59--70, 02 2007.

\bibitem{FrigeriGrasselliRocca}
S.~Frigeri, M.~Grasselli, and E.~Rocca.
\newblock On a diffuse interface model of tumour growth.
\newblock {\em European J. Appl. Math.}, 26(2):215--243, 2015.

\bibitem{Galdi}
G.~P. Galdi.
\newblock {\em An introduction to the mathematical theory of the
  {N}avier-{S}tokes equations}.
\newblock Springer Monographs in Mathematics. Springer, New York, second
  edition, 2011.
\newblock Steady-state problems.

\bibitem{GarckeLam1}
H.~Garcke and K.~F. Lam.
\newblock Global weak solutions and asymptotic limits of a
  {C}ahn-{H}illiard-{D}arcy system modelling tumour growth.
\newblock {\em AIMS Mathematics}, 1(Math-01-00318):318--360, 2016.

\bibitem{GarckeLam3}
H.~Garcke and K.~F. Lam.
\newblock {Well-posedness of a Cahn-Hilliard system modelling tumour growth
  with chemotaxis and active transport}.
\newblock {\em European J. Appl. Math.}, 28(2):284--316, 2017.

\bibitem{GarckeLam4}
H.~Garcke and K.~F. Lam.
\newblock On a {C}ahn-{H}illiard-{D}arcy system for tumour growth with solution
  dependent source terms.
\newblock In {\em Trends in applications of mathematics to mechanics},
  volume~27 of {\em Springer INdAM Ser.}, pages 243--264. Springer, Cham, 2018.

\bibitem{GarckeLamRocca}
H.~Garcke, K.~F. Lam, and E.~Rocca.
\newblock Optimal control of treatment time in a diffuse interface model of
  tumor growth.
\newblock {\em Appl. Math. Optim.}, 78(3):495--544, 2018.

\bibitem{GarckeLam2}
H.~Garcke and K.F. Lam.
\newblock Analysis of a {C}ahn-{H}illiard system with non-zero {D}irichlet
  conditions modeling tumor growth with chemotaxis.
\newblock {\em Discrete Contin. Dyn. Syst.}, 37(8):42--77, 2017.

\bibitem{GarckeLamNuernbergSitka}
H.~Garcke, K.F. Lam, R.~N\"urnberg, and E.~Sitka.
\newblock {A multiphase {C}ahn-{H}illiard-{D}arcy model for tumour growth with
  necrosis}.
\newblock {\em Math. Models Methods Appl. Sci.}, 28(3):525--577, 2018.

\bibitem{GarckeLamSitkaStyles}
H.~Garcke, K.F. Lam, E.~Sitka, and V.~Styles.
\newblock A {C}ahn-{H}illiard-{D}arcy model for tumour growth with chemotaxis
  and active transport.
\newblock {\em Math. Models Methods Appl. Sci.}, 26(6):1095--1148, 2016.

\bibitem{GilardiSprekels}
G.~Gilardi and J.~Sprekels.
\newblock {Asymptotic limits and optimal control for the Cahn–Hilliard system
  with convection and dynamic boundary conditions}.
\newblock {\em Nonlinear Analysis}, 178:1--31, 2019.

\bibitem{HawkinsZeeKristofferOdenTinsley}
A.~Hawkins-Daarud, K.~G. van~der Zee, and J.~T. Oden.
\newblock Numerical simulation of a thermodynamically consistent four-species
  tumor growth model.
\newblock {\em Int. J. Numer. Methods Biomed. Eng.}, 28(1):3--24, 2012.

\bibitem{HilhorstKampmannNguyenZee}
D.~Hilhorst, J.~Kampmann, T.~N. Nguyen, and K.~G. Van Der~Zee.
\newblock Formal asymptotic limit of a diffuse-interface tumor-growth model.
\newblock {\em Math. Models Methods Appl. Sci.}, 25(6):1011--1043, 2015.

\bibitem{HintermuellerWegner}
M.~Hinterm\"uller and D.~Wegner.
\newblock Distributed optimal control of the {C}ahn-{H}illiard system including
  the case of a double-obstacle homogeneous free energy density.
\newblock {\em SIAM J. Control Optim.}, 50(1):388--418, 2012.

\bibitem{JiangWuZheng}
J.~Jiang, H.~Wu, and S.~Zheng.
\newblock Well-posedness and long-time behavior of a non-autonomous
  {C}ahn-{H}illiard-{D}arcy system with mass source modeling tumor growth.
\newblock {\em J. Differential Equations}, 259(7):3032--3077, 2015.

\bibitem{KahleLam}
C.~Kahle and K.F. Lam.
\newblock {Parameter Identification via Optimal Control for a
  Cahn--Hilliard-Chemotaxis System with a Variable Mobility}.
\newblock {\em Appl. Math. Optim.,
  \url{https://doi.org/10.1007/s00245-018-9491-z}}, Mar 2018.

\bibitem{SchaettlerLedzewicz2}
U.~Ledzewicz and H.~Sch\"attler.
\newblock {Multi-input optimal control problems for combined tumor
  anti-angiogenic and radiotherapy treatments}.
\newblock {\em J. Optim. Theory Appl.}, 153(1):195--224, 2012.

\bibitem{OdenTinsleyHawkins}
J.~T. Oden, A.~Hawkins, and S.~Prudhomme.
\newblock General diffuse-interface theories and an approach to predictive
  tumor growth modeling.
\newblock {\em Math. Models Methods Appl. Sci.}, 20(3):477--517, 2010.

\bibitem{Oke}
S.I. Oke, M.B. Matadi, and S.S. Xulu.
\newblock {Optimal Control Analysis of a Mathematical Model for Breast Cancer}.
\newblock {\em Mathematical and Computational Applications}, 23(2), 2018.

\bibitem{PreziosiTosin}
L.~Preziosi and A.~Tosin.
\newblock Multiphase modelling of tumour growth and extracellular matrix
  interaction: mathematical tools and applications.
\newblock {\em Journal of Mathematical Biology}, 58(4):625, Oct 2008.

\bibitem{SchaettlerLedzewicz}
H.~Sch\"attler and U.~Ledzewicz.
\newblock {\em {Optimal control for mathematical models of cancer therapies}},
  volume~42 of {\em Interdisciplinary Applied Mathematics}.
\newblock Springer, New York, 2015.
\newblock An application of geometric methods.

\bibitem{ShibataShimizu}
Y.~Shibata and S.~Shimizu.
\newblock On the {S}tokes equation with {N}eumann boundary condition.
\newblock In {\em Regularity and other aspects of the {N}avier-{S}tokes
  equations}, volume~70 of {\em Banach Center Publ.}, pages 239--250. Polish
  Acad. Sci. Inst. Math., Warsaw, 2005.

\bibitem{Signori}
A.~Signori.
\newblock {Optimal Distributed Control of an Extended Model of Tumor Growth
  with Logarithmic Potential}.
\newblock {\em Appl. Math. Optim.,
  \url{https://doi.org/10.1007/s00245-018-9538-1}}, 2018.

\bibitem{Signori2}
A.~Signori.
\newblock {Vanishing parameter for an optimal control problem modeling tumor
  growth}.
\newblock {\em ArXiv e-prints:
  \href{https://arxiv.org/abs/1903.04930}{arXiv:1903.04930}}, 2019.

\bibitem{Simon}
J.~Simon.
\newblock {Compact sets in the space $L^p(0,T,B)$}.
\newblock {\em Annali di Matematica Pura ed Applicata}, 146(1):65--96, 1986.

\bibitem{Sprekels-Wu}
J.~Sprekels and H.~Wu.
\newblock {Optimal Distributed Control of a Cahn--Hilliard--Darcy System with
  Mass Sources}.
\newblock {\em Appl. Math. Optim.,
  \url{https://doi.org/10.1007/s00245-019-09555-4}}, 2019.

\bibitem{Swan}
G.W. Swan.
\newblock {Role of optimal control theory in cancer chemotherapy}.
\newblock {\em Mathematical Biosciences}, 101(2):237 -- 284, 1990.

\bibitem{triebel}
H.~Triebel.
\newblock {\em {Interpolation Theory, Function Spaces, Differential
  Operators}}.
\newblock North-Holland Publishing Company, 1978.

\bibitem{troeltzsch}
F.~Tröltzsch.
\newblock {\em {Optimal Control of Partial Differential Equations: Theory,
  Methods and Applications, Graduate Studies in Mathematics (Vol. 112)}}.
\newblock Amer. Math. Soc., 2010.

\bibitem{ZhaoLiu1}
X.~Zhao and C.~Liu.
\newblock Optimal control of the convective {C}ahn-{H}illiard equation.
\newblock {\em Appl. Anal.}, 92(5):1028--1045, 2013.

\bibitem{ZhaoLiu2}
X.~Zhao and C.~Liu.
\newblock Optimal control for the convective {C}ahn-{H}illiard equation in 2{D}
  case.
\newblock {\em Appl. Math. Optim.}, 70(1):61--82, 2014.

\end{thebibliography}

%\clearpage

\end{document}